\documentclass{amsart}
\usepackage{a4}
\usepackage{graphics}
\usepackage{color}
\usepackage{bm}
\usepackage{pifont}
\usepackage{amssymb}
\usepackage[all]{xy}
\usepackage{amsmath}
\usepackage{stmaryrd}
\newtheorem{introthm}{Theorem}
\newtheorem{introcor}[introthm]{Corollary}

\newtheorem{theorem}{Theorem}[section]
\newtheorem{lemma}[theorem]{Lemma}
\newtheorem{proposition}[theorem]{Proposition}

\theoremstyle{definition}

\newtheorem{definition}[theorem]{Definition}
\newtheorem{example}[theorem]{Example}

\newtheorem{construction}[theorem]{Construction}

\newtheorem{remark}[theorem]{Remark}
\theoremstyle{remark}

\numberwithin{equation}{section}

\usepackage{hyperref}
\def\Chi{{\mathbb X}}

\def\div{{\rm div}}

\def\reg{{\rm reg}}
\def\rq#1{\widehat{#1}}
\def\t#1{\widetilde{#1}}
\def\b#1{\overline{#1}}
\def\bangle#1{\langle #1 \rangle}

\def\KK{{\mathbb K}}

\def\ZZ{{\mathbb Z}}

\def\WDiv{\operatorname{WDiv}}
\def\PDiv{\operatorname{PDiv}}
\def\id{{\rm id}}

\def\Cl{\operatorname{Cl}}

\def\Spec{{\rm Spec}}

\def\faces{{\rm faces}}

\def\im{{\rm im}}

\begin{document}
 
 \title{Valuative and geometric characterizations of Cox sheaves}
 
 \author[B.~Bechtold]{Benjamin Bechtold} 
\address{Mathematisches Institut, Universit\"at T\"ubingen,
Auf der Morgenstelle 10, 72076 T\"ubingen, Germany}
\email{benjamin.bechtold@uni-tuebingen.de}

\maketitle

\begin{abstract}
 We give an intrinsic characterization of Cox sheaves on Krull schemes in terms of their valuative algebraic properties. 
We also provide a geometric characterization of their graded relative spectra in terms of good quotients of graded schemes, extending the work of \cite{ArDeHaLa} on relative spectra of Cox sheaves on normal varieties. 
Moreover, we obtain an irredundant characterization of Cox rings which in turn produces a normality criterion for certain graded rings. 
\end{abstract}

\section*{Introduction}
Cox sheaves on normal (pre-)varieties $X$ currently are an active field of research with focus on questions of finite generation and explicit calculation of their ring $\mathcal{R}(X)$ of global sections (called the {\em Cox ring}) and quotient constructions describing $X$ in terms of a Cox ring and combinatorics \cite{Ba2011, BeHa2007, Br2007, CaTe2013, Co1995, Ha2008, HaSu2010, HuKe2000, StVe2010}. 
Known properties of Cox rings include triviality of homogeneous units and graded factoriality \cite{Ar2009, ArDeHaLa, BeHa2003}, i.e. factoriality of the monoid of non-zero homogeneous elements. In the case of a free class group $\Cl(X)$ the latter is equivalent to genuine factoriality of the Cox ring $\mathcal{R}(X)$ by a result from \cite{An1979}. 

Our present purpose is to investigate and characterize Cox sheaves in the more general setting of Krull schemes, i.e. schemes with a finite cover by spectra of Krull rings, compare \cite{LeOr1982}. 
A Cox sheaf on $X$ is a $\Cl(X)$-graded $\mathcal{O}_X$-algebra $\mathcal{R}$ with homogeneous components $\mathcal{O}_X(D)$ for all $[D] \in \Cl(X)$, equipped with a {\em natural} multiplication. This translates into the more formal requirement that there exists a 
graded morphism from the {\em divisorial $\mathcal{O}_X$-algebra} associated to $\WDiv(X)$
$$\pi : \mathcal{O}_X(\WDiv(X)) := \bigoplus_{D \in \WDiv(X)}{\mathcal{O}_X(D) \cdot \chi^D} \longrightarrow \mathcal{R}$$ 
 such that each restriction $\pi : \mathcal{O}_X(D) \rightarrow \mathcal{R}_{[D]}$ is an isomorphism, see Section~\ref{characterization-of-Cox-sheaves} for more motivation of this definition. 
 All Cox sheaves on $X$ are linked via the maps from $\mathcal{O}_X(\WDiv(X))$ and thus they have the same graded invariants and their monoids of homogeneous elements are isomorphic modulo units, although the Cox sheaves themselves need not be isomorphic. Furthermore, 
 one Cox sheaf has locally or globally finitely generated sections if and only if all Cox sheaves on $X$ do, see Proposition~\ref{prop:weak-uniqueness-of-cox-sheaves}.

Our main result is an intrinsic characterization of Cox sheaves in terms of their valuative algebraic properties. 
We start by illustrating our terminology. 
Recall that on a Krull scheme $X$ each prime divisor $Y$ defines a discrete valuation $\nu_{Y,X} : \mathcal{K}(X)^* \rightarrow \ZZ$ 
with valuation ring $\mathcal{K}(X)_{\nu_{Y,X}} = \mathcal{O}_{X,Y}$ and $\mathcal{O}_X(U)$ is the intersection over all $\mathcal{O}_{X,Y}$ with $Y \in U$. In terms of sheaves each $Y$ defines a {\em discrete valuation} $\nu_Y : \mathcal{K}^* \rightarrow \imath_Y(\ZZ)$ from the constant sheaf $\mathcal{K}^*$ onto the skyscraper sheaf at $Y$ and $\mathcal{O}_X$ is the intersection 
$$\mathcal{O}_X = \bigcap_{Y {\rm prime}}{\mathcal{K}_{\nu_Y}} \subseteq \mathcal{K}$$
over the corresponding {\em discrete valuation sheaves} $\mathcal{K}_{\nu_Y}$. Thus, $\mathcal{O}_X$ is a {\em Krull sheaf}. 
We will show that Cox sheaves admit a similar description and may even be characterized in such terms. 
Let $K := \Cl(X)$ and let $\mathcal{R}$ be a Cox sheaf. 
The role as an ambient sheaf of $\mathcal{R}$ is taken by the constant $K$-graded sheaf $\mathcal{S}$ assigning the stalk $\mathcal{R}_{\xi}$ at the generic point. Every homogeneous component of $\mathcal{S}$ is of type $\mathcal{K}$ and hence $\mathcal{S}$ is {\em $K$-simple}, i.e every homogeneous section is invertible. 
Each prime divisor $Y$ on $X$ now defines a {\em discrete $K$-valuation} $\mu_Y : \mathcal{S}^+ \rightarrow \imath_Y(\ZZ)$ on the subsheaf $\mathcal{S}^+ \subseteq \mathcal{S}$ of $K$-homogeneous non-zero elements and a corresponding {\em discrete $K$-valuation sheaf} $\mathcal{S}_{\mu_Y} \subseteq \mathcal{S}$ whose sections over $U$ are generated by all $K$-homogeneous elements that are valuated non-negatively by $\mu_{Y,U}$. $\mathcal{R}$ is then the intersection
$$
\mathcal{R} = \bigcap_{Y {\rm prime}}{\mathcal{S}_{\mu_Y}} \subseteq \mathcal{S}
$$
and is thereby a {\em $K$-Krull sheaf}. 
In this terminology, whose precise definitions are given in Section~\ref{divisorial-O_X-algebras}, 
our result is the following.

\begin{introthm}\label{intro-thm:char-of-Cox-sheaves}
 Let $X$ be a Krull scheme with generic point $\xi$, fraction field $\mathcal{K}$, essential valuations $\nu_Y$ for the prime divisors $Y \in X$, and let $K$ be an abelian group and $\mathcal{R}$ a $K$-graded sheaf on $X$. Then $\mathcal{R}$ is isomorphic to a Cox sheaf if and only if the following hold:
 \begin{enumerate}
   \item the constant sheaf $\mathcal{S}$ assigning $\mathcal{R}_{\xi}$ is $K$-simple with $\deg_K(\mathcal{S}^+) = K$ and $\mathcal{S}_0 = \mathcal{K}$, 
  \item $\mathcal{R}$ is a $K$-Krull sheaf in $\mathcal{S}$ defined by discrete $K$-valuations $\{\mu_Y\}_Y$ which restrict to $\{\nu_Y\}_Y$ on $\mathcal{K}^*$, in particular $\mathcal{R}_0 = \mathcal{O}_X$,
  \item $\div_K:= \sum_Y{\mu_Y}$ is surjective and $\div_{K,X}$ has kernel $\mathcal{R}(X)^{+,*} = \mathcal{R}(X)^*_0$. 
 \end{enumerate}
 In this case, 
 the required $\Cl(X)$-grading is provided by the natural 
 isomorphism $K \rightarrow \Cl(X), \deg_K(f) \mapsto [\div_{K,X}(f)]$.
 \end{introthm}
 
 If $\mathcal{R}$ is a Cox sheaf with the required $\Cl(X)$-grading, then the above map is the identity on $\Cl(X)$. 
 This intrinsic characterization of Cox sheaves underlines the fact that they form a natural class of graded sheaves. 
   Properties (i) and (ii) are direct graded analoga of the properties of the structure sheaf, they occur in various graded $\mathcal{O}_X$-algebras e.g. in {\em divisorial $\mathcal{O}_X$-algebras} $\mathcal{O}_X(L)$ of subgroups $L \leq \WDiv(X)$. 
  Property (iii) ensures the correct $\Cl(X)$-grading, the first part being equivalent to surjectivity, the second to injectivity of the canonical map $K \rightarrow \Cl(X)$. 
  The reference for the equation $\mathcal{R}(X)^{+,*} = \mathcal{O}(X)^*$ is \cite{ArDeHaLa}.
  
  In Theorem~\ref{intro-thm:Cox-sheaf-props} below we give further details on Cox sheaves. We briefly explain the occuring graded properties and invariants. 
  A {\em $K$-integral} ring $R$ (i.e. a ring without $K$-homogeneous zero divisors) is {\em $K$-factorial} if the monoid $R^+$ of non-zero homogeneous elements is factorial. 
 The {\em homogeneous fraction ring} $Q^+(R)$ is the localization of $R$ by $R^+$. A {\em $K$-Krull ring} is the graded analogon of a Krull ring. Its {\em essential $K$-valuations} form the minimal family of $K$-valuations defining $R$ in $Q^+(R)$. They correspond bijectively to the {\em $K$-prime divisors}, i.e. the minimal non-zero $K$-prime ideals $\mathfrak{p}$ of $R$. The {\em $K$-valuation ring} of an essential $K$-valuation $\nu_{\mathfrak{p}}$ is the graded localization $R_{\mathfrak{p}}$. More detailed information on $K$-Krull rings is found in Section~\ref{K-Krull-rings}. 
 The essential $K$-valuations of a $K$-Krull sheaf $\mathcal{R}$ are defined in terms of the $K$-Krull rings $\mathcal{R}(U)$ for all affine $U$, see Section~\ref{divisorial-O_X-algebras}. 
 
 \begin{introthm}\label{intro-thm:Cox-sheaf-props} 
 Let $X$ be a Krull scheme. Then each Cox sheaf $\mathcal{R}$ on $X$ is quasi-coherent and has the following properties:
 \begin{enumerate}
  \item For each open $U$, the ring $\mathcal{R}(U)$ is $K$-factorial and $\deg_K(\mathcal{R}(U)^+)$ generates $K$. If $U$ is affine, then $\mathcal{S}(U) = Q^+(\mathcal{R}(U))$ and $\deg_K(\mathcal{R}(U)^+)$ equals $K$. 
  \item $\{\mu_Y\}_Y$ are the essential $K$-valuations of $\mathcal{R}$ and the sections of their valuation sheaves $\mathcal{S}_{\mu_Y}$ 
  are 
  \[
   \mathcal{S}_{\mu_Y}(U) = \left\{
\begin{array}{ll}
 \mathcal{R}_Y & Y \in U \\
 \mathcal{S}(X) & Y \notin U
  \end{array}
   \right. 
  \]
  in particular
  \[
   \mathcal{R}(U) = \bigcap_{Y \in U}{\mathcal{R}_Y} \subseteq \mathcal{S}(X) 
  \]
  \item The stalk at $x \in X$ is the $K$-local $K$-Krull ring
  \[
   \mathcal{R}_x = \bigcap_{x \in \b{\{Y\}}}{\mathcal{S}(X)_{\mu_{Y,X}}} = \bigcap_{x \in \b{\{Y\}}}{\mathcal{R}_Y} \subseteq \mathcal{S}(X)
  \]
 The homogeneous elements of its $K$-maximal ideal $\mathfrak{a}_x$ resp. the homogeneous units of $\mathcal{R}_x$ are
 \begin{align*}
  \mathfrak{a}_x \cap \mathcal{R}_x^+ & = \{f_x \in \mathcal{R}_x^+; \; \text{there ex. }U \ni x, f \in \mathcal{R}(U)^+  \text{ s.th. } x \in |\div_{K,U}(f)|\} \\
  \mathcal{R}_x^{+,*} & = \{f_x \in \mathcal{R}_x^+; \; \text{there ex. }U \ni x, f \in \mathcal{R}(U)^+ \text{ s.th. } \div_{K,U}(f) =0\} 
 \end{align*}
 and $\deg_K(\mathcal{R}_x^{+,*}) \subseteq \Cl(X)$ is the subgroup of classes $[D]$ represented by a divisor $D$ which is principal near $x$. 
  \item For a prime divisor $Y$, each generator of the maximal ideal of $\mathcal{O}_{X,Y} = (\mathcal{R}_Y)_0$ also generates $\mathfrak{a}_Y$, in particular, $\mathcal{R}_Y$ has units in every degree. 
 \end{enumerate}  
 \end{introthm}
    $K$-factoriality of the rings $\mathcal{R}(U)$ is due to surjectivity of $\div_{K,U}$ which is essentially the argument from \cite{Ar2009}. The first proof of $K$-factoriality of Cox rings is due to \cite{BeHa2003}, it is valid for Cox sheaves of finite type on normal prevarieties. 
    Cox sheaves on affine Krull schemes also allow the following description.

 \begin{introcor}
  A $K$-graded sheaf $\mathcal{R}$ on an affine Krull scheme $X$ is a Cox sheaf if and only if it is the sheaf $\mathcal{R} = \t{R}$ associated to a $K$-graded $\mathcal{O}(X)$-algebra $R$ such that: 
  \begin{enumerate}
   \item $R$ is $K$-factorial (in particular, $R$ is a $K$-Krull ring) and $R^{+,*} = R_0^*$, 
   \item $R_0 = \mathcal{O}(X)$, $Q^+(R)_0 = Q(R_0)$ and $\deg_K(Q^+(R)) = K$, 
   \item the essential $K$-valuations of $R$ restrict bijectively on $Q(R_0)$ to the essential valuations of $R_0$, 
  \end{enumerate}

 \end{introcor}

The canonical choice for a geometric realization of a quasi-coherent $K$-graded $\mathcal{O}_X$-algebra $\mathcal{F}$ is its {\em graded relative spectrum} $\Spec_{K,X}(\mathcal{F})$ which is glued from the $K$-spectra (i.e. sets of $K$-prime ideals) of $\mathcal{F}(U)$ for all affine open $U \subseteq X$. 
This object belongs to the category of graded schemes (which contains the category of schemes, i.e. $0$-graded schemes, as a full subcategory) wherein structure sheaves are graded and morphisms between affine graded schemes are comorphisms of maps of graded rings, see Section~\ref{graded-schemes}. 
From the perspective of \cite{De2013} the category of graded schemes is situated between the categories of $\mathbb{F}_1$-schemes and classical schemes within their common parent category of sesquiad schemes, see Remark~\ref{rem:graded-schemes-subcat-of-cong-schemes}. 
Graded algebraic properties of $\mathcal{F}$ naturally correspond to geometric properties of $\Spec_{K,X}(\mathcal{F})$. 
The $\Cl(X)$-graded relative spectrum of a Cox sheaf $\mathcal{R}$ on $X$ together with the canonical morphism $q : \Spec_{\Cl(X),X}(\mathcal{R}) \rightarrow X$ is called its {\em graded characteristic space}. 
Since a Cox sheaf is a $K$-Krull sheaf, its graded characteristic space is a {\em $K$-Krull scheme} which is the generalization of Krull schemes in the category of graded schemes. A $K$-Krull scheme $\rq{X}$ comes with natural notions of {\em $K$-Weil divisors} as sums of {\em $K$-prime divisors} (i.e. points with one-codimensional closure), {\em $K$-principal divisors} for all non-zero homogeneous sections of the sheaf $\mathcal{K}_K$ assigning the stalk $\mathcal{O}_{\rq{X},\rq{\xi}}$ at the generic point, and {\em $K$-class groups}. 
Our second main result is the following geometric characterization of graded characteristic spaces. 

\begin{introthm}\label{intro-thm:char-of-graded-char-spaces}
 Let $q: \rq{X} \rightarrow X$ be a morphism from a $K$-graded scheme to a scheme. 
 Then $X$ is a Krull scheme and $q$ is a graded characteristic space 
 if and only if the following hold:
 \begin{enumerate}
  \item $\rq{X}$ is a $K$-graded $K$-Krull scheme with $\deg_K(\mathcal{K}_K(\rq{X})^+) = K$, 
  \item $q$ is a good quotient and induces a commutative diagram of presheaves 
  \[
   \xy
   \xymatrix
   {
   \mathcal{K}^* \ar[r]^-{\div} \ar[d]^-{\cong}_-{q^*}  &  \WDiv \\
   (q_*\mathcal{K}_K)^*_0 \ar[r]^-{q_*\div^K} & q_*\WDiv^K \ar[u]_-{\substack{\; \\ \rq{Y} \mapsto q(\rq{Y}) }}^-{\cong}
   }
   \endxy
  \]
  \item $\Cl^K(\rq{X}) = 0$, and $\mathcal{O}(\rq{X})^{+,*} = \mathcal{O}(\rq{X})_0^*$. 
\end{enumerate}
If $\rq{X} = \Spec_{K, X}(\mathcal{R})$ with a Cox sheaf $\mathcal{R}$ then with $\div_K := \sum_Y{\mu_Y}$ the following commutative diagram extends the diagram of (ii): 
  \[
   \xy
   \xymatrix
   {
   \mathcal{S}^+ \ar@{->>}[r]^-{\div_K} \ar[d]^-{\cong}_-{q^*}  &  \WDiv \\
   q_*\mathcal{K}_K^+ \ar@{->>}[r]^-{q_*\div^K} & q_*\WDiv^K \ar[u]_-{\substack{\; \\ \rq{Y} \mapsto q(\rq{Y}) }}^-{\cong}
   }
   \endxy
  \]
  For each prime divisor $Y$ the preimage $q^{-1}(Y)$ consists of a single $K$-prime divisor $\rq{Y}$. 
  If $\rq{x} \in \rq{X}$ is the unique point contained in all closures of points mapped to $x \in X$, then $\rq{x} \in \b{\{\rq{Y}\}}$ if and only if $x \in \b{\{Y\}}$. In particular, $\mathcal{O}_{\rq{X},\rq{x}} = \mathcal{R}_x$.
\end{introthm}

This result extends the geometric characterization of relative spectra of Cox sheaves of finite type on normal prevarieties given in \cite{ArDeHaLa}; indeed, with respect to normal prevarieties Theorem~\ref{intro-thm:char-of-graded-char-spaces} allows a translation into terms of good quotients by quasi-torus actions, see Theorem~\ref{thm:char-of-char-spaces-of-prevars}. 
In the following theorem we also generalize their characterization of Cox rings.

\begin{introthm}\label{intro-thm:char-of-Cox-rings}
 If $X$ is a Krull scheme with class group $K$ and Cox ring $R$, then the following hold:
 \begin{enumerate}
  \item $R$ is $K$-factorial, 
  \item $R^{+,*} = R^*_0$, 
  \item $\deg_K(R^+)$ generates $K$, 
  \end{enumerate}
  If $X$ has a cover by affine complements of divisors (e.g. $X$ is separated or of affine intersection), then
  \begin{enumerate}
  \item[(iv)] each localization $R_{\mathfrak{p}}$ at a $K$-prime divisor has units in every degree.  
   \end{enumerate}
   Conversely, if $K$ is finitely generated and $R$ satisfies (i) - (iv), %(i.e. $R$ is an {\em algebraic Cox ring}), 
   then there exists a Krull scheme $X$ of affine intersection with class group $K$ and Cox ring $\mathcal{R}(X) = R$.
\end{introthm}

Here, property (iv) implies property (iii). 
The additional assumption is needed in order to translate the fact that the stalks $\mathcal{R}_Y$ have units in every degree into a property of the global ring $\mathcal{R}(X)$. 
In the case that $R$ is finitely generated over an algebraically closed base field $\KK$, property (iv) translates into freeness of the action of $H = \Spec_{{\rm max}}(\KK[K])$ on a big open subset of $\Spec_{{\rm max}}(R)$ - see Remark~\ref{rem:inv-stalks-and-free-actions}, which is the property featured in the characterization of finitely generated Cox rings of normal prevarieties of affine intersection given by \cite{ArDeHaLa}. 
Our characterization of Cox rings of Krull schemes with cover by affine divisor complements and finitely generated class groups by conditions (i), (ii) and (iv) is irredundant, see Remark~\ref{rem:irredundancy}
Together with normality of Cox rings of normal prevarieties over $\KK$ \cite[Thm. I.5.1.1]{ArDeHaLa} we obtain:
\begin{introcor}
 Let $K$ be a finitely generated abelian group, $\KK$ an algebraically closed field and $R$ a $K$-graded affine $\KK$-algebra satisfying the above properties (i)- (iv). Then $R$ is normal. 
\end{introcor}

The paper is organized as follows: Section~\ref{K-Krull-rings} lays the algebraic foundations for the later parts, introducing $K$-Krull rings and providing the key preparation for the calculation of the essential $K$-valuations of a Cox sheaf $\mathcal{R}$ and hence the $K$-prime divisors of $\Spec_{K,X}(\mathcal{R})$ in Theorem~\ref{thm:alg-divisorial-algebras}. 
In Section~\ref{divisorial-O_X-algebras} we introduce $K$-Krull sheaves and discuss the example of divisorial $\mathcal{O}_X$-algebras $\mathcal{O}_X(L)$ associated to subgroups $L \leq \WDiv(X)$. 
In Section~\ref{characterization-of-Cox-sheaves} we give background on the definition of Cox sheaves and prove Theorem~\ref{intro-thm:char-of-Cox-sheaves}. 
Section~\ref{graded-schemes} offers a first general introduction to graded schemes and good quotients thereof as well as $K$-Krull schemes and their $K$-Weil divisors and class groups. The example of graded spectra of monoid algebras (Example~\ref{ex:graded-spec-of-monoid-algebra}) relates graded schemes to $\mathbb{F}_1$-schemes, i.e. schemes over the field with one element as treated in \cite{De2008}. We also indicate how graded schemes fit into the more general framework of sesquiad schemes from \cite{De2013}, see Remark~\ref{rem:graded-schemes-subcat-of-cong-schemes}. 
In Section~\ref{graded-char-spaces} we prove Theorems~\ref{intro-thm:char-of-graded-char-spaces} and~\ref{intro-thm:char-of-Cox-rings} using Theorem~\ref{intro-thm:char-of-Cox-sheaves}. 
In Section~\ref{graded-schemes-and-diagonalizable-actions} we point out some aspects of the connection between graded schemes of finite type and diagonalizable actions on prevarieties, in particular, we reformulate Theorem~\ref{intro-thm:char-of-graded-char-spaces} in this more familiar setting. This requires the concept of invariant structure sheaves whose stalks naturally encode generic isotropy groups of the action, see Remark~\ref{rem:inv-stalks-and-free-actions}. Furthermore, we provide details on the connection between orbit closures, graded schemes and combinatorics (Remark~\ref{remark:orbit-closures}), and go on to show that the toric graded scheme corresponding to a toric variety is canonically identified with the defining polyhedral fan (Remark~\ref{rem:toric-graded-schemes}). 

The author is grateful for helpful discussions with J\"{u}rgen Hausen. 

\section{$K$-Krull rings}\label{K-Krull-rings}

We start by recalling some generalities and notations from graded Algebra. All rings are taken to be commutative with unit. All abelian groups used to grade rings are written additively. A $K$-graded ring is a ring with a decomposition $R = \bigoplus_{w \in K}{R_w}$ into abelian groups such that $R_w R_{w'} \subseteq R_{w+w'}$. The sets of $K$-homogeneous elements with and without zero, and the group of $K$-homogeneous units are denoted $R^{+,0}$, $R^+$ and $R^{+,*}$, respectively. 
A morphism of graded rings is a map $\phi : R \rightarrow R'$ with accompanying group homomorphism $\psi : K \rightarrow K'$ such that $\phi$ restricts to group homomorphisms $R_w \rightarrow R'_{\psi(w)}$. A morphism of graded rings is called {\em degree-preserving} if the accompanying map is the identity. For any fixed $K$ the category of graded rings has a subcategory of $K$-graded rings with degree-preserving morphisms, and this subcategory has direct and inverse limits. 
A $K$-graded module $M$ over a $K$-graded ring $R$ is a $R$-module with a decomposition $M = \bigoplus_{w \in K}{M_w}$ into abelian groups such that $R_w M_{w'} \subseteq M_{w+w'}$, where the elements of $\bigcup_{w \in K}{M_w}$ are called homogeneous elements. A {\em $K$-graded submodule} of $M$ is a submodule of the form $N = \bigoplus_{w \in K}{N \cap M_w}$, i.e. a submodule generated by homogeneous elements.

\begin{remark}
 If $B \rightarrow R$ is a morphism of graded rings, then $R$ is also called a {\em graded algebra} over $B$. The graded algebras over $B$ form a category with the obvious morphisms. This category has coproducts: If $\phi_R : B \rightarrow R$ and $\phi_S : B \rightarrow S$ are morphisms accompanied by $\psi_L : K \rightarrow L$ and $\psi_M : K \rightarrow M$ then $R \otimes_B S$ is naturally graded by $L \times M / \im(\psi_L \times -\psi_M)$ and the canonical maps $R \rightarrow R\otimes_B S$ and $S \rightarrow R \otimes_B S$ are morphisms of graded algebras over $B$. This statement is used to define fiber products of graded schemes. 
\end{remark}

Classical algebraic properties of rings and their elements and ideals have graded analoga which are obtained by restricting the defining axioms to homogeneous elements resp. graded ideals. In particular, there are natural concepts of graded divisibility theory, i.e. {\em $K$-integrality}, {\em $K$-prime} and {\em $K$-irreducible} elements, {\em $K$-factoriality}, as well as {\em $K$-prime} and {\em $K$-maximal} ideals, {\em $K$-locality}, {\em $K$-noetherianity} etc. Several authors have studied such properties: \cite{Ka1995} treats $K$-prime ideals and invariants of graded modules over $K$-noetherian rings. Graded divisibility theory was introduced and interpreted geometrically by \cite{Ha2008,Ar2009} who showed that $K$-factoriality is a natural property of Cox rings of normal prevarieties. Graded integral closures and their behaviour under coarsening have been studied in \cite{Ro2013}. 
The localization of $R$ by a $K$-prime ideal $\mathfrak{p}$ is denoted $R_{\mathfrak{p}}:= (R^+ \setminus \mathfrak{p})^{-1}R$. 
By localizing a $K$-integral ring $R$ by $R^+$ we obtain $Q^+(R)$, the {\em $K$-homogeneous fraction ring} in which every homogeneous element is invertible, making it {\em $K$-simple}. 
In general, a $K$-graded ring $R$ is $K$-simple if and only if $R_0$ is a field and $\deg_K(R^+)$ is a group, and in that case $r R_0 = R_{\deg_K(r)}$ holds for every $r \in R^+$. 
A $K$-integral ring $R$ is {\em $K$-normal}, if each homogeneous fraction that is integral over $R$ (i.e. over $R^{+,0}$) is an element of $R$.

In the following major part of this section we take a slightly more detailed look 
at $K$-Krull rings, 
the graded equivalent of Krull rings. Proofs of their basic properties may be obtained from the proofs of the respective properties of Krull rings (found e.g. in \cite{Fo1973, LaMCa1971}) by restricting the arguments to homogeneous elements resp. graded ideals. In Example~\ref{ex:A(K,phi)} we treat a canonical class of $K$-Krull rings which form the algebraic analogon of divisorial $\mathcal{O}_X$-algebras $\mathcal{O}_X(L)$ of subgrops of Weil divisors. Theorem~\ref{thm:alg-divisorial-algebras} gives details on their $K$-divisors and essential $K$-valuations and thus provides the key for the calculation of the essential $K$-valuations of Cox sheaves in Section~\ref{characterization-of-Cox-sheaves} and the $K$-Weil divisors of their graded characteristic spaces in Section~\ref{graded-char-spaces}.

\begin{definition}
Let $S$ be a $K$-simple ring. 
A {\em discrete $K$-valuation} on $S$ is a group epimorphism $\nu : S^+ \rightarrow \ZZ$ with $\nu(a + b) \geq \min\{\nu(a), \nu(b)\}$ for all $w \in K$, $a, b \in S_w\setminus 0$ with $a + b\neq 0$. 
Its {\em discrete $K$-valuation ring} is the subring $R_{\nu} \subseteq S$ generated by the preimage of $\ZZ_{\geq 0}$ under $\nu$. 

A {\em $K$-Krull ring} is an intersection $R$ of discrete $K$-valuation rings $R_{\nu_j} \subseteq S$, $j \in J$ such that for each $a \in R^+$  only finitely many $\nu_j(a)$ are non-zero.
\end{definition}

Let $R$ be a $K$-integral ring. 
For $K$-graded $R$-submodules $\mathfrak{a}, \mathfrak{b}$ of $Q^+(R)$, product $\mathfrak{a}\mathfrak{b}$ and quotient $[\mathfrak{a} : \mathfrak{b}] = \{f \in Q^+(R); f \mathfrak{b} \subseteq \mathfrak{a}\}$ are again $K$-graded. 
A {\em $K$-fractional ideal} is a $K$-graded proper $R$-submodule $\mathfrak{a} \leq_R Q^+(R)$ 
with $[R : \mathfrak{a}] \neq 0$.

\begin{construction}
Let $R$ be a $K$-integral ring. A $K$-fractional ideal $\mathfrak{a}$ is called a {\em $K$-divisor} of $R$ if $\mathfrak{a} = [R : [R:\mathfrak{a}]]$. The set $Div^K(R)$ of $K$-divisors of $R$ equipped with the operation sending $\mathfrak{a}$ and $\mathfrak{b}$ to
 \[
  \mathfrak{a} \boldmath{\hm{+}} \mathfrak{b} := [R : [R: \mathfrak{a}\mathfrak{b}]]  
 \]
 and the partial order defined by $\mathfrak{a} \leq \mathfrak{b} :\Leftrightarrow \mathfrak{a} \supseteq \mathfrak{b}$ is a partially ordered semi group with neutral element $R$ in which each two elements have infimum and supremum. 
There is a canonical homomorphism 
\[
 \div^K: Q^+(R)^+ \rightarrow Div^K(R), \; f \mapsto  R f
\]
with kernel $R^{+,*}$ whose image $PDiv^K(R)$ is called the group of {\em $K$-principal divisors}. The cokernel $\Cl^K(R)$ is called the {\em $K$-class semi-group} of $R$. 
\end{construction}

The semi-group of $K$-divisors characterizes $R$ as follows:

\begin{theorem}
Let $R$ a $K$-integral ring. 
\begin{enumerate}
 \item $R$ is a $K$-Krull ring if and only if $Div^K(R)$ is a group and every non-empty set of positive elements in $Div^K(R)$ has a minimal element. 
 \item $R$ is $K$-factorial if and only if $R$ is a $K$-Krull ring with $\Cl^K(R)=0$. 
\end{enumerate}

\end{theorem}

\begin{remark}\label{rem:val-of-divisors}
Let $\{\nu_j\}_{j\in J}$ be a defining family of the $K$-Krull ring $R$. 
For a $K$-fractional ideal $\mathfrak{a}$ and $j \in J$ we set 
\[
 \nu(\mathfrak{a}_j) := \max\{-\nu_j(f);\; f \in [R : \mathfrak{a}] \cap Q^+(R)^+\} \in \ZZ .
\]
This notion is well-defined and satisfies $\nu_j(\mathfrak{a}) = \nu([R: [R: \mathfrak{a}]])$. 
Furthermore, there is a monomorphism of ordered groups
\[
 Div^K(R) \longrightarrow \bigoplus_{j \in J}{\ZZ} , \quad \mathfrak{a} \longmapsto \{\nu_j(\mathfrak{a})\}_{j \in J} .
\] 
\end{remark}

\begin{proposition}
 Let $R$ be a $K$-Krull ring. Then the following hold:
 \begin{enumerate}
  \item The minimal positive $K$-divisors are those that are $K$-prime as ideals in $R$, and these are the minimal non-zero $K$-prime ideals of $R$; they are called the {\em $K$-prime divisors} of $R$ and form a $\ZZ$-basis of $Div^K(R)$, 
  \item For each $K$-prime divisor $\mathfrak{p}$ the map $\nu_{\mathfrak{p}}$ assigning to $a \in Q^+(R)$ the coefficient with which $\mathfrak{p}$ occurs in $\div^K(a)$ is a $K$-valuation on $Q^+(R)$. Its $K$-valuation ring is $R_{\mathfrak{p}}$ and we have $\mathfrak{p} \cap R^+ = \nu_{\mathfrak{p}}^{-1}(\ZZ_{>0}) \cap R^+$. The coefficient of a $K$-divisor $\mathfrak{a}$ at $\mathfrak{p}$ is  
  $$\nu_{\mathfrak{p}}(\mathfrak{a}) = \min\{\nu_{\mathfrak{p}}(a);\; a \in \mathfrak{a} \cap Q^+(R)^+\} .$$ 
  The family $\{\nu_{\mathfrak{p}}\}_{\mathfrak{p}}$ is minimal among all families defining $R$ in $Q^+(R)$, it is called the family of {\em essential $K$-valuations} of $R$. 
 \end{enumerate}

\end{proposition}

\begin{remark}
 Assertion (i) is in particular an existence statement: A $K$-Krull ring has $K$-prime divisors if and only if it is not $K$-simple. A general $K$-integral ring need not have $K$-prime divisors (i.e. minimal non-zero $K$-prime ideals), even if it is not $K$-simple. 
\end{remark}

\begin{proposition}
 A $K$-noetherian $K$-integral ring is a $K$-Krull ring if and only if it is $K$-normal. 
\end{proposition}

Next, we treat the algebraic construction that lies beneath divisorial $\mathcal{O}_X$-algebras of subgroups $L$ of Weil divisors. 

\begin{example}\label{ex:A(K,phi)}
 Let $A$ be a Krull ring with essential valuations $\{\nu_{\mathfrak{p}}\}_{\mathfrak{p}}$ and let $\phi : K \rightarrow Div(A)$ be a homomorphism of abelian groups. The group algebra $S := Q(A)[K]$ is $K$-simple and 
 \begin{align*}
 \mu_{\mathfrak{p}}: S^+ & \longrightarrow \ZZ \\
  a \chi^w & \longmapsto \nu_{\mathfrak{p}}(a) + \nu_{\mathfrak{p}}(\phi(w)) = \nu_{\mathfrak{p}}(\div(a) + \phi(w))
 \end{align*}
defines a $K$-valuation on $S$ for every prime divisor $\mathfrak{p}$. The ring $R = \bigcap_{\mathfrak{p}}{S_{\mu_{\mathfrak{p}}}}$ is a $K$-Krull ring with homogeneous components
\[
  R_w = \{a \in Q(A); \; a = 0 \text{ or } \div(a) + \phi(w) \geq 0\} \cdot \chi^w \quad \text{for }w \in K .
 \]
\end{example}

\begin{theorem}\label{thm:alg-divisorial-algebras}
 In the above notation, the ring $R$ has the following properties:
 \begin{enumerate}
  \item $R_0 = A$ and $Q^+(R)$ is canonically isomorphic to $Q(A)[K]$, 
  \item $\{\mu_{\mathfrak{p}}\}_{\mathfrak{p}}$ are the essential $K$-valuations of $R$ 
  and there is an isomorphism 
  \begin{align*}
   Div^K(R) & \longrightarrow Div(A) \\
   \alpha: \mathfrak{b} & \longmapsto  \begin{Large}\textbf{$\sum$}\end{Large}_{\mathfrak{p}}{\mu_{\mathfrak{p}}(\mathfrak{b}) \mathfrak{p}}  \\
   [R: [R: R\mathfrak{a}]]  & \longmapsfrom \mathfrak{a} : \beta
  \end{align*}
  which restricts to a bijection
    \begin{align*}
  \{K-\text{prime divisors of }R\} & \longrightarrow \{\text{prime divisors of }A\} \\
   \mathfrak{q} & \longmapsto \mathfrak{q} \cap A \\
  \bangle{\mu_{\mathfrak{p}}^{-1}(\ZZ_{> 0}) \cap R} & \longmapsfrom \mathfrak{p} 
  \end{align*}
  and induces an isomorphism $\Cl_K(R) \cong \Cl(A)/\b{\im(\phi)}$,
  \item The localization $R_{\mathfrak{q}}$ by a $K$-prime divisor $\mathfrak{q}$ of $R$ has units in every degree with $(R_{\mathfrak{q}})_0 = A_{\mathfrak{p}}$ where $\mathfrak{p} = \mathfrak{q} \cap A$, 
  \item $R$ has homogeneous components of every $K$-degree, i.e. $\deg_K(R^+) = K$.
 \end{enumerate}

\end{theorem}

\begin{lemma}
In the above situation, let $\mathfrak{a} \subset Q(A)$ be a fractional ideal and $\mathfrak{p}$ a prime divisor of $A$. Then
 \[ 
  \mu_{\mathfrak{p}}(R \mathfrak{a}) = \nu_{\mathfrak{p}}(\mathfrak{a}) ,
 \]
 in particular, $\mu_{\mathfrak{p}}(R \mathfrak{p}') = \delta_{\mathfrak{p},\mathfrak{p}'}$ holds for any prime divisor $\mathfrak{p}'$ of $A$. 

\begin{proof}
 We calculate 
 \begin{align*}
  \mu_{\mathfrak{p}}(R \mathfrak{a}) %& = \max\{ \mu_{\mathfrak{p}}(r);\; r \in Q^+(R)^+, \; R\mathfrak{a} \subseteq R r\} \\
  & \geq \max\{ \mu_{\mathfrak{p}}(a);\; a \in Q(A)^*, \; \mathfrak{a} \subseteq A a\} %\\
  %& = \max\{ \nu_{\mathfrak{p}}(a);\; a \in Q(A)^*, \; \mathfrak{a} \subseteq A a\}  
  = \nu_{\mathfrak{p}}(\mathfrak{a}) %\\
  %& = \min\{\nu_{\mathfrak{p}}(a);\; a\in \mathfrak{a} \} 
  = \min\{\mu_{\mathfrak{p}}(a);\; a\in \mathfrak{a} \} \\
  & \geq \max\{ \mu_{\mathfrak{p}}(r);\; r \in Q^+(R)^+, \; R\mathfrak{a} \subseteq R r\} = \mu_{\mathfrak{p}}(R \mathfrak{a}) .
 \end{align*}
\end{proof}
\end{lemma}

\begin{proof}[Proof of Theorem~\ref{thm:alg-divisorial-algebras}]
Assertion (iv) follows from the Approximation Theorem for Krull rings (see e.g. \cite{Fo1973}). 
For (i) we observe that the canonical morphism $Q^+(R) \rightarrow Q(A)[K]$ is surjective. Indeed, for every $w \in K$ there exists an element $0 \neq a \chi^w \in R_w$ by (iv) and thus $a \chi^w / a \chi^0$ is mapped to $\chi^w$. 

For (ii) we first observe that by Remark~\ref{rem:val-of-divisors} $\alpha$ is a monomorphism of partially ordered groups. 
Since the above Lemma and Remark~\ref{rem:val-of-divisors} give $\alpha(\beta(\mathfrak{a})) = \mathfrak{a}$, $\alpha$ is also surjective and $\beta$ is its inverse map. 
In particular, they induce bijections between the sets of minimal positive elements of $Div^K(R)$ and $Div(A)$. This means that the divisorial ideals $\beta(\mathfrak{p})$ are the $K$-prime divisors of $R$. 
For any prime divisor $\mathfrak{p}$ of $A$ and any $K$-prime divisor $\mathfrak{q}' = \beta(\mathfrak{p}')$ we have $\mu_{\mathfrak{p}}(\mathfrak{q}') = \delta_{\mathfrak{p}, \mathfrak{p}'}$ and therefore $\mu_{\mathfrak{p}} = \nu_{\beta(\mathfrak{p})}$ is the essential $K$-valuation corresponding to $\mathfrak{q} = \beta(\mathfrak{p})$ and we have $\mathfrak{q} = \bangle{\mu_{\mathfrak{p}}^{-1}(\ZZ_{> 0}) \cap R}$. Since $\nu_{\mathfrak{q}}$ restricts to $\nu_{\mathfrak{p}}$ this implies $\mathfrak{q} \cap A = \mathfrak{p}$. 

 For (iii), let $\mathfrak{q}$ be a $K$-prime divisor of $R$, and let $\mathfrak{p} = \mathfrak{q} \cap A$ be the corresponding prime divisor of $A$. 
 First we show $(R_{\mathfrak{q}})_0 = A_{\mathfrak{p}}$. Let $f/g \in (R_{\mathfrak{p}})_0$. Then there are $a, b \in A$ with $f/g = a/b$, and $\nu_{\mathfrak{p}}(a/b) = \nu_{\mathfrak{q}}(f/g) \geq 0$, i.e. $a/b \in A_{\mathfrak{p}}$. 
 Each generator $t_0$ of the maximal ideal of $A_{\mathfrak{p}}$ satisfies $\nu_{\mathfrak{q}}(t_0) = \nu_{\mathfrak{p}}(t_0) = 1$ and is thus a generator of the $K$-maximal ideal of $R_{\mathfrak{q}}$. This implies the assertion. 
\end{proof}

\begin{remark}
 From \cite{ElKuWa2004} we observe that if $K$ is free, then for a $K$-prime divisor $\mathfrak{q}$ of $R$ and $\mathfrak{p} = A \cap \mathfrak{q}$, each choice of a generator for the maximal ideal of $A_{\mathfrak{p}}$ gives an isomorphism $R_{\mathfrak{q}} \cong A_{\mathfrak{p}}[K]$. In particular, $R_{\mathfrak{q}}$ is a Krull ring and one concludes that $R$ is a Krull ring. 
\end{remark}

The following well-behaved class of graded morphisms will be used in Section~\ref{characterization-of-Cox-sheaves} to describe the relation between the sections of $\mathcal{O}_X(\WDiv(X))$ and Cox sheaves. Their properties, some of which are listed in Proposition~\ref{prop:compw-iso-epis} below, ensure that Cox sheaves inherit all graded properties from $\mathcal{O}_X(\WDiv(X))$. 

A {\em component-wise isomorphic epimorphism} (CIE) is an epimorphism of graded rings $\phi: R' \rightarrow R$ accompanied by an epimorphism $\psi : K' \rightarrow K$ that restricts to isomorphisms on homogeneous components. 
If $\psi$ is fixed, then for each given $R$ one obtains $R'$ and $\phi$ constructively and uniquely by setting $R'_{w'}:=R_{\psi(w')}$ for $w' \in K'$. The functor from $K$-graded rings to $K'$-graded rings thus defined is right adjoint to the coarsening functor associated to $\psi$, see \cite{Ro2012}.
If $R'$ and $\psi$ are fixed, then each homomorphism $\chi: \ker(\psi) \rightarrow R'^{+,*}$ with $\chi(w') \in R'_{w'}$ defines a CIE $\phi : R' \rightarrow R' / \bangle{1 - \chi(w'); w' \in \ker(\psi)}$. 

\begin{proposition}\label{prop:compw-iso-epis}
 Let $\phi : R'\rightarrow R$ be a CIE. Then the following hold:
 \begin{enumerate}
  \item $R^+ \cong R'^+/\phi^{-1}(1)$ and there is a bijection of sets of graded ideals 
  \begin{align*}
   \{\mathfrak{a}' \trianglelefteq R' \} & \longrightarrow \{\mathfrak{a} \trianglelefteq R\} \\
   \mathfrak{a}' & \longmapsto \phi(\mathfrak{a}') \\
   \bangle{\phi^{-1}(\mathfrak{a}) \cap R'^+} & \longmapsfrom \mathfrak{a}
  \end{align*}
 respecting inclusions, products, quotients, sums and intersections. 
 \item Likewise, there is a bijection between the graded $R'$- resp. $R$-modules of $Q^+(R')$ and $Q^+(R)$.   
 \item If $M \subseteq R^+$ is a submonoid and $M':=\phi^{-1}(M)$, then $M'^{-1}R' \rightarrow M^{-1}R$ is again a CIE, in particular, for every $f' \in R'^+$ the map $R'_{f'} \rightarrow R_{\phi(f')}$ is a CIE. 
 \item $R'$ is $K'$-integral/-simple/-factorial resp. has units in every $K'$-degree if and only if $R$ is $K$-integral/-simple/-factorial resp. has units in every $K$-degree. 
 \item Let $\phi_S : S' \rightarrow S$ be a CIE with $R' \subseteq S'$ and $R \subseteq S$ extending the CIE $\phi: R' \rightarrow R$. 
 Suppose that $S'$ is $K'$-simple. Then the following hold: 
 \begin{enumerate}
  \item $S' = Q^+(R')$ if and only if $S = Q^+(R)$. 
  \item Each $K'$-valuation $\nu$ on $S'$ with $\ker(\phi) \subseteq \ker(\nu)$ induces a $K$-valuation $\b{\nu}$ on $S$ and vice versa. 
  \item $R'$ is a $K'$-Krull ring defined by $\{\nu_j\}_{j \in J}$ if and only if $R$ is a Krull ring defined by $\{\b{\nu}_j\}_{j \in J}$. Here, $\{\nu_j\}_{j \in J}$ are the essential $K'$-valuations if and only if $\{\b{\nu}_j\}_{j \in J}$ are the essential $K$-valuations. 
 \end{enumerate}
 \end{enumerate}
 Thus, $R'$ and $R$ share all of the graded properties defined in terms of graded ideals and all properties of $R^+/R^{+,*}$. 
\end{proposition}

\section{divisorial $\mathcal{O}_X$-algebras}\label{divisorial-O_X-algebras}
Let $K$ be an abelian group. 
We begin with the prerequisites on graded sheaves needed for the definition of $K$-Krull sheaves - the sheaf-theoretic analogon of $K$-Krull rings. 
Recall that a $K$-graded (pre-)sheaf of rings $\mathcal{F}$ on a topological space $X$ is a (pre-)sheaf of $K$-graded rings with degree-preserving restriction maps. $\mathcal{F}$ is also called a graded (pre-)sheaf with grading group $K = gr(\mathcal{F})$. As a presheaf $\mathcal{F}$ then equals $\bigoplus_{w \in K}{\mathcal{F}_w}$ where $\mathcal{F}_w \subseteq \mathcal{F}$ is the (pre-)sheaf of abelian groups assigning $\mathcal{F}(U)_w$ to $U$. 
The monoid of $K$-homogeneous elements of $\mathcal{F}$ is the sheaf of monoids $\mathcal{F}^{+,0} := \bigcup_{w \in K}{\mathcal{F}_w}$. If all $\mathcal{F}(U)$ are $K$-integral, then $\mathcal{F}^+$ denotes the sheaf of monoids given by $\mathcal{F}^{+,0}(U) \setminus \{0\}$. 
A morphism $\phi : \mathcal{G} \rightarrow \mathcal{F}$ of graded (pre-)sheaves comes with a group homomorphism $\psi : gr(\mathcal{G}) \rightarrow gr(\mathcal{F})$ such that each of the morphisms of graded rings $\phi_U$ is accompanied by $\psi$. 
$\mathcal{F}$ is also called a graded $\mathcal{G}$-algebra. 

\begin{definition}
\begin{enumerate}
\item A {\em discrete value sheaf} is a sheaf of abelian groups $\mathcal{Z}$ with values in 
  $\{0, \ZZ\}$ such that $\mathcal{Z}_{\geq 0}(U) := \mathcal{Z}(U)_{\geq 0}$ defines a subsheaf of $\mathcal{Z}$. 
 \item Let $K$ be an abelian group. 
Let $\mathcal{S}$ be a 
sheaf of $K$-simple rings on $X$. 
A {\em discrete $K$-valuation} on $\mathcal{S}$ is a morphism $\mathcal{\nu} : \mathcal{S}^+ \rightarrow \mathcal{Z}$ to a discrete value sheaf such that each $\mathcal{\nu}_U$ is surjective and either a discrete $K$-valuation or zero. 
The associated {\em $K$-valuation sheaf} is the graded subsheaf $\mathcal{S}_{\mathcal{\nu}} \subseteq \mathcal{S}$ of rings generated by $\mathcal{\nu}^{-1}(\mathcal{Z}_{\geq 0})$. 
\item  A {\em $K$-Krull sheaf} in $\mathcal{S}$ is an intersection $\mathcal{R} = \bigcap_{j \in J}{\mathcal{S}_{\mathcal{\nu}_j}}$ of $K$-valuation sheaves such that for every $U$ and every $f \in \mathcal{R}^+(U)$ only finitely many $\mathcal{\nu}_{j, U}(f)$ are non-zero. 
 
 If $X$ is a scheme (or a graded scheme, see Section~\ref{graded-schemes}), then a defining family of $K$-valuations of a $K$-Krull sheaf $\mathcal{R}$ is called the family of {\em essential} $K$-valuations, if for each affine $U \subseteq X$, the family $\{\mathcal{\nu}_{j,U}; j \in J, \mathcal{\nu}_{j,U} \not\equiv 0\}$ is the family of essential $K$-valuations of the $K$-Krull ring $\mathcal{R}(U)$.  
\end{enumerate}
\end{definition}

For $K =0$ we usually omitt the prefix $K$. 
Recall that $X$ is a Krull scheme if it has a finite cover by spectra of Krull rings.   
A prime divisor is a point $Y \in X$ 
with one-codimensional closure. Sums and intersections with subscript $Y$ are taken over all prime divisors $Y$ of $X$ unless specified otherwise. 

\begin{example}
 Let $X$ be a Krull scheme and $K =0$. 
 Each prime divisor defines a valuation $\nu_Y : \mathcal{K}^* \rightarrow \ZZ_Y$ to the skyscraper sheaf $\imath_Y(\ZZ)$. The sections of its valuation sheaf $\mathcal{K}_{\nu_Y}$ on $U$ are $\mathcal{O}_{X,Y}$ if $Y \in U$ and $\mathcal{K}$ otherwise. This turns the structure sheaf 
 \[
  \mathcal{O}_X = \bigcap_Y{\mathcal{K}_{\nu_Y}} 
 \]
into a Krull sheaf with essential valuations $\nu_Y$.  
\end{example}

\begin{remark}
 A quasi-compact scheme $X$ is a Krull scheme if and only if $\mathcal{O}_X$ is a Krull sheaf. 
\end{remark}

For the remainder of this section, $X$ is a Krull scheme. 
Recall that the presheaf of Weil divisors is $\WDiv := \bigoplus_Y{\ZZ_Y}$ 
and there is a morphism 
$$\div := \sum_Y{\nu_Y} : \mathcal{K}^* \rightarrow \WDiv .$$
Its image $\PDiv$ is the presheaf of principal divisors of $X$ and its cokernel is the presheaf $\Cl$ of divisor class groups. For each prime divisor $Y$ there is a natural projection $pr_Y : \WDiv \rightarrow \imath_Y(\ZZ)$. 
The support $|D|$ of a Weil divisor $D \in \WDiv(U)$ is the union over the closures of $\{Y\}$ in $U$, where $Y$ runs through the prime divisors occuring with non-zero coefficient in $D$. 

\begin{example}\label{ex:O_X(L)}
 For a subgroup $L \leq \WDiv(X)$ the constant sheaf of $L$-graded group algebras $\mathcal{S} := \mathcal{K}[L] = \bigoplus_{D \in L}{\mathcal{K} \cdot \chi^D}$ is a sheaf of $L$-simple rings. Each prime divisor $Y \in X$ defines an $L$-valuation 
 \begin{align*}
  \mu_Y : \mathcal{S}^+ & \longrightarrow \imath_Y(\ZZ) \\
  \mathcal{S}^+(U) \ni f \chi^D & \longmapsto \mu_{Y,U}(f \chi^D) := \nu_{Y,U}(f) + pr_{Y,U}(D_{|U})
 \end{align*}
 Then $\mathcal{O}_X(L) := \bigcap_Y{\mathcal{S}_{\mu_Y}}$ is a $L$-Krull sheaf on $X$, called the {\em divisorial $\mathcal{O}_X$-algebra} associated to $L$. 
 Its homogeneous parts have sections $\mathcal{O}_X(L)_D(U) = \mathcal{O}_X(D)(U) \cdot \chi^D$, where $\mathcal{O}_X(D)$ is the $\mathcal{O}_X$-submodule of $\mathcal{K}$ associated to $D$ with sections
 \[
 \mathcal{O}_X(D)(U) = \{f \in \mathcal{K}(U); f=0 \text{ or } \div_U(f) + D_{|U} \geq 0\} , 
\]
 in particular $\mathcal{O}_X(L)_0 = \mathcal{O}_X$. 
 The sum over all $\mu_Y$ defines a morphism 
 \begin{align*}
  \div_L := \sum_Y{\mu_Y} : \mathcal{S}^+ & \longrightarrow \WDiv \\ 
   \mathcal{S}^+(U) \ni f \chi^D & \longmapsto \div_{L,U}(f \chi^D)= \div_U(f) + D_{|U}
 \end{align*}
with kernel $\mathcal{O}_X(L)^{+,*}$. In particular, $\mathcal{O}_X(L)(X)^{+,*}$ is the set of all elements $f \chi^D$ with $\div_X(f) = -D$. 
\end{example}

\begin{proposition}\label{prop:divisorial-O_X-algebras}
 In the above notation, the divisorial algebra $\mathcal{R}:=\mathcal{O}_X(L)$ has the following properties: 
 \begin{enumerate}
  \item there is a canonical isomorphism $\mathcal{S}(X) \cong \mathcal{R}_{\xi}$, and for affine $U$ we have $\mathcal{S}(U) \cong Q^+(\mathcal{R}(U))$ and $\deg_L(\mathcal{R}(U)^+) = L$, 
  \item $\{\mu_Y\}_Y$ are the essential $L$-valuations of $\mathcal{R}$, 
  \item The sections of $\mathcal{S}_{\mu_Y}$ on $U$ equal $\mathcal{R}_Y$ if $Y \in U$ and $\mathcal{S}$ otherwise, 
  \item The stalk at $x \in X$ is the $L$-local $L$-Krull ring 
  \begin{align*}
   \mathcal{R}_x = \bigcap_{x \in \b{\{Y\}}}{\mathcal{S}(X)_{\mu_{Y,X}}} \subseteq \mathcal{S}(X)
  \end{align*}
whose $L$-maximal ideal $\mathfrak{a}_x$ has homogeneous elements  
   \[
    \mathfrak{a}_x \cap \mathcal{R}_x^+ = \{ g \in \mathcal{R}_x^+; \; \text{there is }U \ni x \text{with }g \in \mathcal{R}(U), x \in |\div_{L,U}(g)|\} .
   \]
  Its homogeneous units are 
  \begin{align*}
    \mathcal{R}_x^{+,*}  &=  \{ g \in \mathcal{S}(X)^+; \; \text{ there is } U \ni x \text{ with } \div_{L,U}(g)  =0 \} \\
    & = \bigcap_{x \in \b{\{Y\}}}{\ker(\mu_{Y,X})} 
  \end{align*}
  and $\deg_L(\mathcal{R}_x^{+,*})$ is the subgroup of Weil divisors in $L$ that are principal near $x$. 
   The stalk at a prime divisor $Y$ has units in every degree and $\mathfrak{a}_Y$ has a generator in $(\mathcal{R}_Y)_0 = \mathcal{O}_{X,Y}$. 
  \item The image of $\div_{L,U}$ (resp. ${\div_{L,U}}_{|\mathcal{R}(U)^+}$) consists of (the non-negative elements in) the union over all $\Cl(U)$-classes of divisors in $L_{|U}$, in particular, if $L_{|U}$ maps onto $\Cl(U)$, then $\mathcal{R}(U)$ is $L$-factorial, 
 \end{enumerate}

\end{proposition}

\begin{remark}\label{rem:O_X(L)-localization}\cite[Remark I.3.1.6]{ArDeHaLa}
For an open set $U \subset X$ each $g \in \mathcal{O}_X(L)(U)^+$ defines an open subset $U_g := U \setminus |\div_{L,U}(g)|$ and a canonical isomorphism
 \[
 \mathcal{O}_X(L)(U)_g \cong \mathcal{O}_X(L)(U_g) .
 \]
In particular, $\mathcal{O}_X(L)$ is quasi-coherent. 
\end{remark}

\begin{remark}
As a subsheaf of a constant sheaf, $\mathcal{O}_X(L)$ has injective restriction maps, and thus all canonical maps $\mathcal{O}_X(L)(U) \rightarrow \mathcal{O}_X(L)_x$ for $x \in U$ and $\mathcal{O}_X(L)_x \rightarrow \mathcal{O}_X(L)_{x'}$ for $x \in \b{x'}$ are injectve as well.
\end{remark}

\begin{proof}[Proof of Proposition~\ref{prop:divisorial-O_X-algebras}]
 For (i) note that the inclusions $\imath_V : \mathcal{R}(V) \subseteq \mathcal{S}(X)$ induce an injection $\imath_{\xi}: \mathcal{R}_{\xi} \rightarrow \mathcal{S}(X)$. For affine open $U \subseteq X$ we consider the canonical monomorphism 
  \[
  \alpha_U: Q^+(\mathcal{R}(U)) \rightarrow \mathcal{S}(X), \quad g/h \mapsto \imath_U(g) \imath_U(h)^{-1} .
 \]
 For surjectivity of $\imath_{\xi}$ and $\alpha_U$ let $f \chi^D \in S(X)_D$ and set $W := X \setminus |\div_X(f) + D|$. Then $f_{|W} \chi^D \in \mathcal{R}(W)_D$ and $\imath_{\xi}((f_{|W} \chi^D)_{\xi}) = f \chi^D$. 
 Furthermore, there exists $h \in \mathcal{O}_X(U) = \mathcal{R}(U)_0$ with $U_{h} \subseteq W \cap U$, and by Remark~\ref{rem:O_X(L)-localization} there are $m > 0$ and $g \chi^D\in \mathcal{R}(U)_D$ such that $g \chi^D (h \chi^0)^{-m} = f_{|U_h} \chi^D$ and hence $\alpha_U(g \chi^D /(h \chi^0)^{m}) = f \chi^D$. 
For the supplement and assertion (ii), we invoke Theorem~\ref{thm:alg-divisorial-algebras} with $A := \mathcal{O}(U)$, $\phi : L \rightarrow L_{|U}$ and $R = \mathcal{R}(U)$. 

For (iii) note that if $Y \notin U$ then $\mu_{Y,U}=0$ and therefore $\mathcal{S}(U)_{\mu_{Y,U}} = \mathcal{S}(U)$. 
If $Y \in U$, then (iv) gives $\mathcal{R}_Y = \mathcal{S}(X)_{\mu_{Y,X}} = \mathcal{S}(U)_{\mu_{Y,U}}$. 

For (iv) first note, that in $\mathcal{S}(X)$ we have
\[
 \mathcal{R}_x^+ = \{ f \chi^D \in \mathcal{S}(X)^+; \text{ there is } U \ni x \text{ with }f \in \mathcal{O}_X(D)(U)\}
\]
If $f \in \mathcal{O}_X(D)(U)$ and $x \in U$, then for every prime divisor $Y$ containing $x$ in its closure we have $Y \in U$ and thus $\mu_{Y,X}(f \chi^D) = \mu_{Y,U}(f \chi^D) \geq 0$. Conversely, if $f \chi^D \in \mathcal{S}(X)^+$ satisfies $\mu_{Y,X}(f \chi^D) \geq 0$ for all prime divisors $Y$ with $x \in \b{\{Y\}}$, then for the complement $W$ of all prime divisors $Y'$ with $\mu_{Y',X}(f \chi^D) < 0$ we have $f \in \mathcal{O}_X(D)(W)$ and $x \in W$. 
This establishes that $\mathcal{R}_x$ is the $L$-Krull ring in $\mathcal{S}(X)$ defined by all $\mu_{Y,X}$ with $x \in \b{\{Y\}}$. 
Its homogeneous units are therefore obtained as the intersection of the kernels of the defining $L$-valuations. For the second representation, let $g \in \mathcal{R}_x^{+,*} \subseteq \mathcal{S}(X)^+$ and let $W$ be the complement of all prime divisors $Y'$ with $\mu_{Y',X}(g) \neq 0$. None of these $Y'$ contain $x$ in their closure, therefore $x \in W$. The equation $\div_{L,W}(g)=0$ holds by definition of $W$. 
Conversely, if $g \in \mathcal{S}(X)^+$ satisfies $\div_{L,U}(g) = 0$ for some $U$ containing $x$, then $g$ is invertible in $\mathcal{R}(U)$ and hence in $\mathcal{R}_x$. 
In particular, $\deg_L(\mathcal{R}_x^{+,*})$ is contained in the subgroup of Weil divisors in $L$ that are principal near $x$. For the converse inclusion, let $D_{|U} = \div_U(f)$. Then $f^{-1}\chi^D \in \mathcal{R}(U)_D$ is a unit and thus $(f^{-1} \chi^D)_x$ is a unit. 
Let $\mathfrak{a}_x$ be the ideal generated by the set 
\[
 \{ g \in \mathcal{R}_x^+; \; \text{there is }U \ni x \text{with }g \in \mathcal{R}(U), x \in |\div_{L,U}(g)|\} 
\]
Due to Lemma~\ref{lem:div_K-and-sums}, this set is closed under addition of elements of the same degree and therefore coincides with $\mathfrak{a}_x \cap \mathcal{R}_x^+$. Its complement in $\mathcal{R}_x^+$ is $\mathcal{R}_x^{+,*}$ and thus $\mathfrak{a}_x$ is the only $L$-maximal ideal of $\mathcal{R}_x$.

For the supplement on the stalks of prime divisors first note that $(\mathcal{R}_Y)_0 = \mathcal{O}_{X,Y}$ because taking stalks commutes with direct sums. Now, let $U \subseteq X$ be affine with $Y \in U$ and $D \in L$. By the Approximation Theorem for Krull rings there exists $f \in \mathcal{O}_X(D)(U)$ with $pr_Y(\div_U(f)) =0$ and $pr_{Y'}(\div_U(f)) = pr_{Y'}(D_{|U})$ for all $Y' \in |D_{|U}|$ and $pr_{Y''}(\div_U(f)) \geq 0$ for all other prime divisors. Then $(f \chi^D)_Y$ is a unit of degree $D$ in $R_Y$.

In (v) the statements on the image of $\div_{L,U}$ follow from the definition of $\div_L$. 
If $\div_{L,U}$ is surjective, then $\mathcal{R}(U)^+ / \mathcal{R}(U)^{+,*} \cong \WDiv_{\geq 0}(U)$ is a factorial monoid, and thus also $\mathcal{R}(U)^+$ is a factorial monoid, meaning that $\mathcal{R}(U)$ is $L$-factorial. 
\end{proof}

\begin{remark}
 By arguments from \cite{ElKuWa2004}, each $\mathcal{R}(U)$ is a Krull ring. Thus, if $L_{|U}$ maps onto $\Cl(U)$, then $\mathcal{R}(U)$ is a factorial ring by \cite{An1979}. 
 However, the sections of Cox sheaves will in general not be factorial. Integrality and normality for the sections of Cox sheaves are proven in the case of normal (pre-)varieties, see \cite[Sect. I.5.1]{ArDeHaLa} or \cite{BeHa2003}, but neither proof seems to be applicable in the more general setting of Krull schemes. 
\end{remark}

\section{characterization of Cox sheaves}\label{characterization-of-Cox-sheaves}
Intuitively, a Cox sheaf should be a $\Cl(X)$-graded $\mathcal{O}_X$-algebra $\mathcal{R}$ whose $[D]$-homogeneous parts are of type $\mathcal{O}_X(D)$, i.e. there should exist isomorphisms of $\mathcal{O}_X$-modules $\pi_D : \mathcal{O}_X(D) \rightarrow \mathcal{R}_{[D]}$ for all Weil divisors $D \in \WDiv(X)$. This requirement fixes the $\mathcal{O}_X$-module structure of Cox sheaves. There is however no canonical way to equip such an $\mathcal{O}_X$-module with an $\mathcal{O}_X$-algebra structure. 
But we can and do require that the multiplication in $\mathcal{R}$ is {\em natural} in the sense that up to the isomorphisms $\pi_D$ it is given by the multiplication in $\mathcal{K}$; meaning that to multiply homogeneous sections of degree $[D]$ and $[D']$ in $\mathcal{R}$ is the same as to apply $\pi_D^{-1}$ resp. $\pi_{D'}^{-1}$, multiply the resulting sections of $\mathcal{O}_X(D)$ and $\mathcal{O}_X(D')$ in $\mathcal{K}$ and then apply $\pi_{D+D'}$. 
This translates into the condition that the morphism of $\mathcal{O}_X$-modules 
\[
 \mathcal{O}_X(\WDiv(X)) = \bigoplus_{D \in \WDiv(X)}{\mathcal{O}_X(D) \cdot \chi^D} \xrightarrow{\quad \pi \quad} \mathcal{R} = \bigoplus_{[D] \in \Cl(X)}{\mathcal{R}_{[D]}} .
\]
defined by the sum of the $\pi_D$ is a morphism of graded $\mathcal{O}_X$-algebras. 
Summing up, a Cox sheaf is defined as a $\Cl(X)$-graded $\mathcal{O}_X$-algebra $\mathcal{R}$ posessing a graded morphism from $\mathcal{O}_X(\WDiv(X))$ to $\mathcal{R}$ which restricts to isomorphisms of the homogeneous parts. 
Such a kind of morphism between two graded sheaves has useful properties and thus justifies the following definition.

\begin{definition}
 A morphism $\pi: \mathcal{F} \rightarrow \mathcal{G}$ of graded presheaves of rings with accompanying map $\psi : L \rightarrow K$ of abelian groups is called a {\em component-wise isomorphic epimorphism} (CIE) if $\psi$ is an epimorphism and the restriction $\pi_{|\mathcal{F}_w}$ is an isomorphism for every $w \in L$. Equivalently, every pair $(\pi_U, \psi)$ is a CIE of rings. 
 \end{definition}
 
 \begin{remark}\label{rem:compw-iso-epis-of-presheaves}
 Let $\psi: L \rightarrow K$ be an epimorphism of abelian groups. 
 \begin{enumerate}
  \item Let $\pi: \mathcal{F} \rightarrow \mathcal{G}$ be a CIE accompanied by $\psi$. Then there is an exact sequence 
 \[
  \xy
  \xymatrix
  {
  0 \ar[r] & \ker(\psi) \ar[r]^-{\chi} & \mathcal{F}^+ \ar[r]^-{\pi} & \mathcal{G}^+ \ar[r] & 0 
  }
  \endxy
 \]
 where $\ker(\psi)$ is considered as a constant presheaf and $\chi_U$ maps $w \in L$ to the unique preimage of $1_{\mathcal{G}(U)}$ in $\mathcal{F}(U)_w$. Furthermore, $\mathcal{F}$ is a sheaf if and only if $\mathcal{G}$ and $\ker(\pi)$ are sheaves. 
 \item A morphism $\pi: \mathcal{F} \rightarrow \mathcal{G}$ of graded sheaves of rings is a CIE of sheaves if and only if every $\pi_x : \mathcal{F}_x \rightarrow \mathcal{G}_x$ is a CIE of graded rings. 
 \item Let $\mathcal{F}$ be a $L$-graded presheaf of rings. For a morphism $\chi : \ker(\psi) \rightarrow \mathcal{F}^{+,*}$ with $\chi_U(w) \in \mathcal{F}(U)_w$ the cokernel of the $K$-graded presheaf of ideals $$\mathcal{I}_{\chi}: U \mapsto \bangle{1_{\mathcal{F}(U)} - \chi_U(w); w \in \ker(\psi)}$$ is a CIE. Conversely, every CIE with prescribed $\mathcal{F}$ and $\psi$ is of this form with $\chi$ as in (i). 
 \end{enumerate}
 \end{remark}

 In this terminology, the precise definition of Cox sheaves is the following.
 
\begin{definition}
 Let $X$ be a Krull scheme. A Cox sheaf on $X$ is a $\Cl(X)$-graded sheaf of rings $\mathcal{R}$ such that there exists a CIE $\pi: \mathcal{O}_X(\WDiv(X)) \rightarrow \mathcal{R}$ that is accompanied by the canonical map $\WDiv(X) \rightarrow \Cl(X)$. $\mathcal{R}$ is then automatically a $\mathcal{O}_X$-algebra with $\mathcal{R}_0 = \mathcal{O}_X$. 
\end{definition}

\begin{remark}
 For a $\Cl(X)$-graded sheaf $\mathcal{R}$ on a Krull scheme $X$ the following are equivalent:
 \begin{enumerate}
  \item $\mathcal{R}$ is a Cox sheaf, 
  \item for every subgroup $L \leq \WDiv(X)$ mapping onto $\Cl(X)$ there exists a CIE $\pi : \mathcal{O}_X(L) \rightarrow \mathcal{R}$, 
  \item for some subgroup $L \leq \WDiv(X)$ mapping onto $\Cl(X)$ there exists a CIE $\pi : \mathcal{O}_X(L) \rightarrow \mathcal{R}$. 
 \end{enumerate}
\begin{proof}
 Assume that (iii) holds. 
 Let $\pi: \mathcal{O}_X(L) \rightarrow \mathcal{R}$ be a CIE. Let $D_j, j\in J$ be a basis of $\WDiv(X)$. Then there exist $D'_j \in L, j \in J$ and $f_j \in \mathcal{O}_X(D'_j - D_j)(X)$ with $\div(f_j) + D'_j =D_j$, and the isomorphisms 
 \[
  \mathcal{O}_X(D_j) \xrightarrow{\cdot f_j} \mathcal{O}_X(D'_j) 
 \]
fit together to a homomorphism $\Phi : \mathcal{O}_X(\WDiv(X)) \rightarrow \mathcal{O}_X(L)$ with accompanying homomorphism $\phi : \WDiv(X) \rightarrow L, D_j \mapsto D'_j$. 
The composition $\pi \circ \Phi : \mathcal{O}_X(L) \rightarrow \mathcal{R}$ is the epimorphism requested in assertion (i). 
\end{proof}

\end{remark}

Existence of Cox sheaves follows from Remark~\ref{rem:compw-iso-epis-of-presheaves} because for any $L$ mapping onto $\Cl(X)$ a suitable map $\chi$ is defined by assigning arbitrary $f_j\in \mathcal{O}_X(D_j)(X)$ with $\div_X(f_j) = -D_j$ to the elements of a basis $\{D_j\}_{j\in J}$ of $L \cap \PDiv(X)$, and thus the presheaf $\mathcal{R} = \mathcal{O}_X(L) / \mathcal{I}_{\chi}$ is a Cox sheaf, compare \cite{ArDeHaLa}. As seen by this discussion, our definition is the axiomatic version of the constructive approach of Hausen and Arzhantsev. 

Uniqueness is a matter of caution. 
In the case that $\Cl(X)$ is free, it is well-known that all Cox sheaves are isomorphic. 
A further condition enforcing uniqueness in the case of prevarieties over an algebraically closed field $\KK$ is $\mathcal{O}(X)^* = \KK^*$ which holds e.g. if $X$ is projective, see \cite[Sect. I.4.3]{ArDeHaLa}. 
In general Cox sheaves are only {\em weakly unique} in the following sense: 

\begin{proposition}\label{prop:weak-uniqueness-of-cox-sheaves}
Let $X$ be a Krull scheme and let $\mathcal{R}$ and $\mathcal{R}'$ be Cox sheaves on $X$ and let $U \subseteq X$ be open. Then the following hold:
\begin{enumerate}
 \item $\mathcal{R}^+/\mathcal{R}^{+,*} \cong \mathcal{O}_X(\WDiv(X))^+/\mathcal{O}_X(\WDiv(X))^{+,*} \cong \mathcal{R}'^+ / \mathcal{R}'^{+,*}$ . 
 \item There are bijections respecting sums, intersections, inclusions, products and ideal quotients between the sets of graded ideals of $\mathcal{R}(U)$ and $\mathcal{R}'(U)$.
 \item $\mathcal{R}(U)$ is finitely generated as a $\mathcal{O}_X(U)$-algebra if and only if $\mathcal{R}'(U)$ is so. If $X$ is a scheme over $S = \Spec(B)$, then $\mathcal{R}(U)$ is finitely generated over $B$ if and only if $\mathcal{R}'(U)$ is so. 
\end{enumerate}

\begin{proof}
Everything but the last assertion follows directly from Proposition~\ref{prop:compw-iso-epis}. 
Suppose that $\mathcal{R}(U)$ is finitely generated by homogeneous sections $g_1,\ldots,g_m$. Then $\Cl(X)$ is finitely generated by Theorem~\ref{intro-thm:char-of-Cox-sheaves}. Let $L \leq \WDiv(X)$ be a finitely generated subgroup mapping onto $\Cl(X)$ and let $\pi : \mathcal{O}_X(L) \rightarrow \mathcal{R}$ and $\pi' : \mathcal{O}_X(L) \rightarrow \mathcal{R}'$ be CIE. Let $\chi$ be the kernel character of $\pi$ and let $D_1,\ldots,D_n$ be a basis of $L \cap \PDiv(X)$. Then $\mathcal{O}_X(L)(U)$ is generated by $\chi_U(\pm D_1),\ldots,\chi_U(\pm D_n)$ and any choice of homogeneous preimages under $\pi_U$ $f_1,\ldots,f_m$ for $g_1,\ldots,g_m$. Thus, $\mathcal{R}'(U)$ is generated by their images under $\pi'_U$.
\end{proof}
  
\end{proposition}

The above shows that the question of uniqueness is of little practical consequence since all Cox sheaves on a given $X$ behave in the same way. 
We now proceed with the proof of Theorem~\ref{intro-thm:char-of-Cox-sheaves}. The general ideal for showing that Cox sheaves have the asserted properties is to show that they are inherited from $\mathcal{O}_X(\WDiv(X))$ because CIEs preserve most 
graded properties (even in both directions). 
 The second part of the proof adepts the arguments of \cite[Thm. I.6.4.3 and Prop. I.6.4.5]{ArDeHaLa}.

\begin{proof}[Proof of Theorem~\ref{intro-thm:char-of-Cox-sheaves}]
 Let $\mathcal{R}$ be a Cox sheaf on $X$ and let $\pi : \mathcal{O}_X(L) \rightarrow \mathcal{R}$ be a CIE. 
 Then we have commutative diagrams with CIE downward arrows 
 \[
  \xy
  \xymatrix{
   \ar@{}^{(1)}[d] & \mathcal{O}_X(L)(U) \ar[r]\ar[d]_{\pi_U} & \mathcal{O}_X(L)_x \ar[d]^{\pi_x} & & \mathcal{O}_X(L)_x \ar[r]\ar[d]_{\pi_U} & \mathcal{O}_X(L)_{x'} \ar[d]^{\pi_{x'}} & \ar@{}_{(2)}[d]\\
  &\mathcal{R}(U) \ar[r] & \mathcal{R}_x & & \mathcal{R}_x \ar[r] & \mathcal{R}_{x'} & 
  }
  \endxy
 \]
 and the lower arrows inherit injectivity from the upper ones. 
 Considering diagram $(1)$ for the generic point $\xi$ we see that $\mathcal{R}_{\xi}$ is $K$-simple with degree zero part $\mathcal{O}_{X,\xi}$ and $\deg_K(\mathcal{R}_{\xi}^+) = K$ because $\mathcal{O}_X(L)_{\xi}$ is $L$-simple with degree zero part $\mathcal{O}_{X,\xi}$ and $\deg_L(\mathcal{O}_X(L)_{\xi}^+) = L$. Denote by $\mathcal{S}_L$ the constant sheaf assigning $\mathcal{O}_X(L)_{\xi}$. Then $\pi_{\xi}$ defines a CIE of sheaves $\pi_{\mathcal{S}}: \mathcal{S}_L \rightarrow \mathcal{S}$ and the constant sheaf $\mathcal{S} := \underline{\mathcal{R}_{\xi}}$ has the desired properties. $\mathcal{R} $ is a subsheaf of $\mathcal{S}$ and hence $K$-integral. 

 For (ii), consider the diagram with exact rows and injective upward arrows 
 \[
  \xy
  \xymatrix
  {
  0 \ar[r] & \ker(\psi) \ar[r]^-{\chi_{\mathcal{S}}} & \mathcal{S}_L^+ \ar[r]^-{\pi_{\mathcal{S}}} & \mathcal{S}^+ \ar[r] & 0 \\
  0 \ar@{=}[u] \ar[r] & \ker(\psi) \ar@{=}[u] \ar[r]^-{\chi} & \mathcal{O}_X(L)^+ \ar@{^{(}->}[u] \ar[r]^-{\pi} & \mathcal{R}^+ \ar[r] \ar@{^{(}->}[u] & 0 \ar@{=}[u] 
  } 
  \endxy
 \]
 Because $\ker(\pi_{\mathcal{S},X}) = \im(\chi_{\mathcal{S},X})$ is contained in $\mathcal{O}_X(L)(X)^{+,*}$, its elements are trivially valuated by all $\mu_{Y,X}$. Consequently, $\ker(\pi_{\mathcal{S},U}) = \im(\chi_{\mathcal{S},U})$ is valuated trivially by all $\mu_{Y,U}$. 
 Thus, each $\mu_Y$ induces a $K$-valuation $\b{\mu}_Y : \mathcal{S}^+ \rightarrow \ZZ_Y$ which also restrict to $\nu_Y$ on $\mathcal{K}^*$. Since $\mathcal{O}_X(L)$ is defined in $\mathcal{S}_L$ by the family $\{\mu_Y\}_Y$, the equality $\mathcal{R} = \bigcap_Y{\mathcal{S}_{\b{\mu}_Y}}$ now follows by applying Proposition~\ref{prop:compw-iso-epis}(v) to the sections over arbitrary open $U$ and assertion (ii) is proven. 
 
 The first part of assertion (iii) is due to the fact that $\mu_Y$ and $\b{\mu_Y}$ have the same image. For the second part, consider an element $\pi_X(f \chi^D)$ of the kernel of $\div_{K, X} = \sum_Y{\b{\mu}_Y}$, i.e. a global homogeneous unit. Then $\div_{L, X}(f \chi^D) = 0$, i.e. $D = -\div_X(f)$ is a principal divisor, hence $\pi_X(f \chi^D)$ has degree $[D] = [0]$.
 
  Concerning the supplement we calculate
  \[
   [D] = \deg_K(\pi_{\xi}(f \chi^D)) = [\deg_L(f \chi^D)] = [\div_{L,X}(f \chi^D)] = [\div_{K,X}(\pi_{\xi}(f \chi^D))] .
  \]
  
It remains to show that a $K$-graded sheaf $\mathcal{R}$ satisfying (i)-(iii) is a Cox sheaf. 
Recall the notation $\div_K := \sum_Y{\b{\mu}_Y} : \mathcal{S}^+ \rightarrow \WDiv$ and note that by (ii) we have ${\div_K}_{|\mathcal{K}^*} = \div$. 
In order to show that $K$ is canonically isomorphic to $\Cl(X)$ we first note that by (ii) the degree map $\deg_K$ induces an isomorphism 
$\mathcal{S}(X)^+ / \mathcal{S}(X)_0^* \cong K$. 
The homomorphism $\delta : \mathcal{S}(X)^+ \rightarrow \Cl(X), f \mapsto [\div_{K,X}(f)]$ thus induces the map 
$$
\b{\delta} : K \longrightarrow \Cl(X), \quad \deg_K(f) \longmapsto [\div_{K,X}(f)]
$$
which has cokernel $\Cl(X)/\b{\im(\div_{K,X})}$ and kernel $\deg_K(\mathcal{R}(X)^{+,*})$. 
Thus, condition (iii) precisely says that $\b{\delta}$ is an isomorphism. 

Next, we show that the ismorphism $\mathcal{K} = \mathcal{S}_0 \xrightarrow{\lambda_f} \mathcal{S}_{\deg_K(f)}$ given by multiplication with $f \in \mathcal{S}(X)^+$ restricts to an isomorphism $\lambda_f : \mathcal{O}_X(\div_{X,K}(f)) \longrightarrow \mathcal{R}_{\deg_K(f)}$. 
Indeed, for a non-zero $g \in \mathcal{O}_X(\div_{K,X}(f))(U)$ we calculate 
\[
 \div_{U,K}(f_{|U} g) = \div_U(g) + (\div_{X,K}(f))_{|U} \geq 0 ,
\]
i.e. $f_{|U} g \in \mathcal{R}(U)^+$. 
Conversely, each non-zero $h \in \mathcal{R}(U)_{\deg_K(f)}$ satisfies
\[
 \div_U((f_{|U})^{-1}h) + {\div_{X,K}(f)}_{|U} = \div_{K,U}(h) \geq 0.
\]

Now, let $L \leq \WDiv(X)$ be any subgroup mapping onto $\Cl(X)$ and let $\{D_j\}_{j\in J}$ be a basis of $L$. Then there exist $f_j \in S^+$ , $j\in J$ with $\div_{K,X}(f_j) = D_j$. 
We set $f_D := \prod_{j\in J}{f_j^{m_j}}$ for $D = \sum_{j\in J}{m_j D_j}$. 
By our first claim we know that $\b{\delta}(\deg_K(f_D)) = [D]$. 
By our second claim, every $f_D$ defines an isomorphism
\[
 \lambda_{f_D} : \mathcal{O}_X(D) \rightarrow \mathcal{R}_{[D]} .
\]
By construction, the isomorphisms respect the multiplication of homogeneous components in $\mathcal{O}_X(L)$ and thus define a graded epimorphism $\mathcal{O}_X(L) \rightarrow \mathcal{R}$. 
Hence, $\mathcal{R}$ is a Cox sheaf. 
\end{proof}

\begin{remark}
One may argue that the characterizing condition (ii) of Theorem~\ref{intro-thm:char-of-Cox-sheaves} is not purely intrinsic because the condition that $\mathcal{R}$ is a Cox sheaf in $\mathcal{S}$ contains an existence statement for certain discrete $K$-valuations which are not derived from $\mathcal{R}$. 
A more obviously intrinsic characterization is obtained by replacing (ii) with 
\begin{enumerate}
\item[(ii')] $\mathcal{R}$ is a sheaf of $K$-Krull rings and every affine open $U$ satisfies:
\begin{enumerate}
 \item $Q^+(\mathcal{R}(U)) = \mathcal{S}(U)$, 
 \item the essential $K$-valuations $\{\mu_{Y,U}\}_{Y \in U}$ of $\mathcal{R}(U)$ restrict on $\mathcal{K}(U)^* = Q(\mathcal{O}_X(U))^*$ to the essential valuations $\{\nu_{Y,U}\}_{Y \in U}$ of $\mathcal{O}_X(U)$, 
 \item for every $Y \in U$ and every affine open $V \ni Y$ we have $\mu_{Y,U} = \mu_{Y,V}$. 
\end{enumerate}
\end{enumerate}
Since the family of essential $K$-valuations of a $K$-Krull ring $R$ may be derived from $R$, condition (ii') is indeed intrinsic. 
\end{remark}

\begin{proof}[Proof of Theorem~\ref{intro-thm:Cox-sheaf-props}]
 Quasi-coherence of $\mathcal{R}$ follows from the more general observation that 
 for every non-zero $g \in \mathcal{R}(U)^+$ there is a canonical isomorphism 
 \[
   \mathcal{R}(U)_g \cong \mathcal{R}(U_g) ,
 \]
 where $U_g := U \setminus |\div_{U,K}(g)|$. This in turn follows directly from the corresponding observation on $\mathcal{O}_X(L)$ using that $\pi_U$ and $\pi_{U_g}$ are CIEs and Proposition~\ref{prop:compw-iso-epis}(iii). 
  
 For (i) note that $K$-factoriality of $\mathcal{R}(U)$ follows via Proposition~\ref{prop:compw-iso-epis}(iv) from $L$-factoriality of $\mathcal{O}_X(U)$ which was proven in Proposition~\ref{prop:divisorial-O_X-algebras}(v). Now, let $D = D^+ - D^-$ be any Weil divisor on $X$, written as a difference of effective divisors. Choose a subgroup $L$ which contains $D^+$ and $D^-$ mapping onto $\Cl(X)$ and a componentwise isomorphic epimorphism $\pi : \mathcal{O}_X(L) \rightarrow \mathcal{R}$. Then for all open $U$ we have
$$
[D] %= [D^+] - [D^-] 
= \deg_K(\pi_U(\chi^{D^+})) - \deg_K(\pi_U(\chi^{D^-})) \in \bangle{\deg_K(\mathcal{R}(U)^+)} .
$$ 
 For the case that $U$ is affine, observe that in the diagram of CIEs
 \[
 \xy
  \xymatrix{
  \mathcal{O}_X(L)(U) \ar[r]\ar[d] & Q^+(\mathcal{O}_X(L)(U)) \ar[r]\ar[d] & \mathcal{O}_X(L)_{\xi} \ar[d] \\
  \mathcal{R}(U) \ar[r] & Q^+(\mathcal{R}(U)) \ar[r] & \mathcal{R}_{\xi} 
  }
 \endxy 
 \]
 the lower right arrow is a graded isomorphism because the upper right arrow is one. Furthermore, 
 $$
 \deg_K(\mathcal{R}(U)^+) = \b{\deg_L(\mathcal{O}_X(L)(U)^+)} = \b{L} = K .
 $$
 
 The first part of assertion (ii) follows directly from Proposition~\ref{prop:divisorial-O_X-algebras}(ii) and Proposition~\ref{prop:compw-iso-epis}(v). The second statement is obvious for $Y \notin U$ and follows from (iii) otherwise. 
 
 For assertion (iii), consider the diagram (2) of inclusions and CIEs from the beginning of the proof of Theorem~\ref{intro-thm:char-of-Cox-sheaves} with $x' = \xi$. Since $\mathcal{O}_X(L)_x$ is the intersection over all $\mathcal{K}[L](X)_{\mu_{Y,X}}$ with $x \in \b{\{Y\}}$, Proposition~\ref{prop:compw-iso-epis}(v) implies that $\mathcal{R}_x$ is the intersection over all $\mathcal{S}(X)_{\b{\mu}_{Y,X}}$ with $x \in \b{\{Y\}}$. Moreover, $\mathcal{R}_x^{+,*}$ is the image of $\mathcal{O}_X(L)_x^{+,*}$ and thus has the requested description. Furthermore, the unique $L$-maximal ideal of $\mathcal{O}_X(L)_x$ is mapped onto a unique $K$-maximal ideal $\mathfrak{a}_x$ by Proposition~\ref{prop:compw-iso-epis}(i). Its homogeneous elements, which were calculated in Proposition~\ref{prop:divisorial-O_X-algebras}(iv), are mapped onto the homogeneous elements of $\mathfrak{a}_x$ which establishes the desired description. 
 
 For (iv), observe that the stalk $\mathcal{R}_Y$ at a prime divisor $Y$ has units in every degree because $\mathcal{O}_X(L)_Y$ does. 
\end{proof}

\begin{remark}
 One of the starting points for the present considerations on the valuative structure of Cox sheaves was \cite[Sect. I.5]{ArDeHaLa}. The $[D]$-divisor of a non-zero $f \in \mathcal{R}(X)_{[D]}$ defined there is $\div_{\Cl(X), X}(f)$ in our notation, and it is shown that the assignment $f \mapsto \div_{[D]}(f)$ is homomorphic and encodes the divisibility relation in $\mathcal{R}(X)$, which in our setting is due to the definition of $\div_{\Cl(X),X}$ as the sum of all $\mu_{Y,X}$. 
\end{remark}

\section{graded schemes}\label{graded-schemes}
Graded schemes are implicitely already well-known from the proj construction, see Example~\ref{ex:proj}. Related examples of graded schemes in the context of toric good quotients have been studied in \cite{Pe2007}. 
More generally, the categories of $K$-graded resp. noetherian $\ZZ$-graded schemes with degree-preserving morphisms have been discussed in \cite{Ca2006, Te2004}. 
The category of graded schemes introduced below includes the aforementioned and has more morphisms, in particular good quotients, which are affine morphisms from $K$-graded to $0$-graded schemes satisfying a natural condition on their structure sheaves, see Definition~\ref{def:graded-good-quotients}. 
Good quotients behave very naturally in that they respect intersections of closed sets and are surjective with distinguished points in each fibre, see Proposition~\ref{prop:graded-good-quotients}. 
We also introduce the other concepts needed for Theorem~\ref{intro-thm:char-of-graded-char-spaces}, namely $K$-Krull schemes which are the most general objects with well-behaved notions of Weil and principal divisors. 

\begin{definition}
 The {\em $K$-spectrum} of a $K$-graded ring $R$ is the set $X:=\Spec_K(R)$ of $K$-prime ideals of $R$, endowed with the topology whose closed sets are of the form $V(\mathfrak{a}) = \{\mathfrak{p} \in X; \mathfrak{a} \subseteq \mathfrak{p}\}$ with $K$-graded ideals $\mathfrak{a} \trianglelefteq R$.  
 Its $K$-graded structure sheaf $\mathcal{O}_X$ (with $K$-local stalks) is defined on the basis $X_f := X \setminus V(\bangle{f})$ of principal open sets for $f \in R^{+,0}$ by $\mathcal{O}_X(X_f) := R_f$ and on arbitrary open $U \subseteq X$ by
 \[
  \mathcal{O}_X(U) := \varprojlim_{X_f \subseteq U}{\mathcal{O}_X(X_f)} .
 \]
 The pair $(\Spec_K(R),\mathcal{O}_{\Spec_K(R)})$ is the {\em affine $K$-graded scheme} corresponding to $R$. 
 
A {\em graded scheme} is a pair $(X, \mathcal{O}_X)$ consisting of a topological space $X$ and a graded sheaf of rings $\mathcal{O}_X$ that has a cover by affine $gr(\mathcal{O}_X)$-graded schemes $(U, {\mathcal{O}_X}_{|U})$. $(X, \mathcal{O}_X)$ is also called a $gr(\mathcal{O}_X)$-graded scheme. 
 A {\em morphism} of the graded schemes $(X, \mathcal{O}_X)$ and $(X', \mathcal{O}_{X'})$ is a continuous map $\phi : X \rightarrow X'$ together with a morphism of graded sheaves $\phi^* : \mathcal{O}_{X'} \rightarrow \phi_*\mathcal{O}_X$ 
 such that for each $x \in X$ the induced graded homomorphism $\phi_x^*: \mathcal{O}_{X',\phi(x)} \rightarrow \mathcal{O}_{X,x}$ satisfies $\bangle{(\phi_x^*)^{-1}(\mathfrak{m}_x) \cap \mathcal{O}_{X',\phi(x)}^+} = \mathfrak{m}_{\phi(x)}$, where $\mathfrak{m}_x$ and $\mathfrak{m}_{\phi(x)}$ are the respective unique $gr(\mathcal{O}_X)$-/$gr(\mathcal{O}_{X'})$-maximal ideals. 
\end{definition}

Schemes are the same as $0$-graded schemes, they form a full subcategory of the category of graded schemes. 
We will often only write $X$ for the graded scheme $(X, \mathcal{O}_X)$. When talking about a morphism of graded schemes we will also write the continuous map $\phi : X \rightarrow X'$ of the underlying topological spaces in place of the pair $(\phi, \phi^*)$. 

\begin{remark}
For an affine $K$-graded scheme $X = \Spec_K(R)$, the stalk at $\mathfrak{p} \in X$ is the graded localization $R_{\mathfrak{p}}$. Morphisms between affine graded schemes are given by maps of graded rings: If $\phi : X \rightarrow X'$ is a morphism of affine graded schemes, 
then $\phi^*: R' = \mathcal{O}(X') \rightarrow R =\mathcal{O}(X)$ is a graded morphism, and $\phi$ maps the point $\mathfrak{p} \in X$ to the point $\bangle{(\phi^*)^{-1}(\mathfrak{p}) \cap R'^+}$. 
\end{remark}

\begin{example}\label{ex:graded-spec-of-monoid-algebra}
 Let $M$ be an abelian monoid 
 contained in an abelian group $K$ and let $k$ be a field. Let $R:= k[M]$ be the canonically $K$-graded monoid algebra of $M$ over $k$. Recall that a {\em face} of an abelian monoid is a submonoid $\tau \subseteq M$ such that $w + w' \in \tau$ implies $w, w' \in \tau$ for all $w,w' \in M$. Then there is an order reversing bijection
 \begin{align*}
  \faces(M) & \longleftrightarrow \Spec_K(R) =:X  \\
  \tau & \longmapsto \mathfrak{p}_{\tau} := \bangle{\chi^w; w \in M \setminus \tau} \\
  \deg_K(R^+ \setminus \mathfrak{p}) =: \tau_{\mathfrak{p}} & \longmapsfrom \mathfrak{p}
 \end{align*}
where $\Spec_K(R)$ is ordered by inclusion. Furthermore, $\mathcal{O}_X(X_{\chi^w}) = k[M - \ZZ_{\geq 0} w]$ and $\mathcal{O}_{X, \mathfrak{p}} = k[M -\tau_{\mathfrak{p}}]$. If $\psi: K' \rightarrow K$ is a group homomorphism mapping the submonoid $M'$ into $M$, then it induces a graded map $\t{\psi} : R':= k[M'] \rightarrow R$ and a map 
$$\faces(M) \rightarrow \faces(M'), \tau \mapsto \psi^{-1}(\tau) \cap M' .$$
The corresponding map of graded schemes is
\[
 \phi: X \longrightarrow X':=\Spec_{K'}(R'), \quad \mathfrak{p}_{\tau} \longmapsto \mathfrak{p}'_{\psi^{-1}(\tau) \cap M'} .
\]
This consideration links graded schemes to combinatorics when applied to finitely generated monoids. On the other hand, we observe that $\tau \subseteq M$ is a face if and only if $\tau^c \sqcup \{\infty\} \subseteq M \sqcup \{\infty\}$ is a prime ideal in the sense of \cite{De2008}. 
Moreover, a subgroup $\mathfrak{p}$ of a $K$-graded ring $R$ generated by homogeneous elements is a $K$-prime ideal if and only if $R^{+,0} \setminus \mathfrak{p}$ is a face of the multiplicative monoid $R^{+,0}$. 
Thus, there is also a canonical homeomorphism between $\Spec_K(R)$ and the space of prime ideals of $M \sqcup \{\infty\}$ which is a scheme over the field $\mathbb{F}_1$ with one element in the sense of \cite{De2008}. This line of thought is continued in Remark~\ref{rem:toric-graded-schemes}. 
\end{example}

\begin{remark}\label{rem:graded-schemes-subcat-of-cong-schemes}
 The category of graded rings is a subcategory of the category of {\em monoidal pairs} or {\em sesquiads} from \cite{De2013} via the assignment $R = \bigoplus_{w \in K}{R_w} \mapsto (R^{+,0}, R)$. Graded ideals of $R$ correspond to ideals of $(R^{+,0}, R)$ and graded localizations of $R$ correspond to localizations of $(R^{+,0}, R)$. Therefore, the affine graded scheme $(\Spec_K(R), \mathcal{O}_{\Spec_K(R)})$ is naturally identified with the set of prime ideals of $(R^{+,0}, R)$ equipped with the structure sheaf induced by $(R^{+,0}, R)$, and the category of graded schemes becomes a subcategory of the category of {\em sesquiad Zariski schemes} which is obtained by gluing prime spectra of sesquiads. 
 Within this category (non-trivially) graded schemes take an intermediate position between schemes, whose affine charts are given by monoid pairs of the form $(R, R)$, and $\mathbb{F}_1$-schemes, whose affine charts are given by pairs of the form $(M, \mathbb{Z}[M])$. 
\end{remark}

Having introduced the core definitions we will from now on assume that basic concepts like {\em generic points}, open and closed graded {\em subschemes} and {\em embeddings} have been introduced as well, naturally extending the well-known notions for schemes. Furthermore, {\em $K$-reduced} resp. {\em $K$-integral} $K$-graded schemes are defined by the absence of homogeneous nilpotent elements resp. homogeneous zero divisors in all sections.  
For a $K$-integral $K$-graded scheme $X$, the constant sheaf $\mathcal{K}_K$ assigning the stalk $\mathcal{O}_{X,\xi}$ at the generic point is a sheaf of $K$-simple rings. 
A quasi-compact $K$-graded scheme is {\em $K$-noetherian} resp. {\em $K$-Krull} if the sections $\mathcal{O}_X(U)$ are $K$-noetherian resp. $K$-Krull for all (or equivalently, some cover of $X$ by) affine $U$. For a $K$-Krull scheme $X$, a {\em $K$-prime divisor} is a point with one-codimensional closure. 

\begin{remark}
 If $X$ is a $K$-Krull scheme, then the stalks at $K$-prime divisors are $K$-valuation rings. Thus, each $K$-prime divisor $Y$ defines a $K$-valuation $\nu_Y : \mathcal{K}_K^+ \rightarrow \imath_Y(\ZZ)$ and these are the essential $K$-valuations of the $K$-Krull sheaf $\mathcal{O}_X = \bigcap_Y{{\mathcal{K}_K}_{\nu_Y}}$. Therefore, a $K$-graded scheme $X$ is a $K$-Krull scheme if and only if it is quasi-compact and $\mathcal{O}_X$ is a $K$-Krull sheaf. 
\end{remark}

If $X$ is a $K$-Krull scheme, then the direct sum over the skyscraper sheaves $\imath_Y(\ZZ)$ is the presheaf $\WDiv^K$ of {\em $K$-Weil divisors} 
and the direct sum over the $K$-valuations $\nu_Y$ defines a morphism $\div^K : \mathcal{K}_K^+ \rightarrow \WDiv^K$. Its image $\PDiv^K$ is the presheaf of {\em $K$-principal divisors} and its cokernel is the presheaf of {\em $K$-divisor class groups} $\Cl^K$. 
The {\em support} $|D|$ of a Weil divisor $D \in \WDiv(U)$ is the union over all $K$-prime divisors $Y \in U$ occuring with non-zero coefficient in $D$. 

\begin{proposition}
 Let $X$ be a $K$-Krull scheme. Then the stalk of $\mathcal{O}_X$ at $x$ is 
 \[
  \mathcal{O}_{X,x} = \bigcap_{x \in \b{\{Y\}}}{(\mathcal{O}_{X,\xi})_{\nu_{Y,X}}} \subseteq \mathcal{O}_{X,\xi} 
 \]
 where $Y$ runs through all $K$-prime divisors $Y$ containing $x$ in their closure. 
\end{proposition}

Since the category of graded rings over a fixed graded ring has the graded tensor product as a coproduct, the category of graded schemes has fiber products. A graded scheme $X$ is called {\em separated} resp. of {\em affine intersection} if the diagonal morphism $\Delta_X : X \rightarrow X \times X$ is a closed embedding resp. affine. The latter property is equivalent to affineness of the intersection of any two affine opens subsets. %Separated graded schemes are of affine intersection. 

\begin{proposition}\label{prop:divisor-complements}
 In a $K$-Krull scheme $X$ of affine intersection the following hold: 
 \begin{enumerate}
  \item Every open affine $U \subseteq X$ is the complement of a $K$-divisor on $X$.
    \item If $X$ is affine, then $V(f) = |\div^K(f)|$ for every $f \in \mathcal{O}(X)^+$.
 \end{enumerate}
%  \begin{proof}
%   Firstly, we observe that for a $f \in \mathcal{O}(X)^+$ the $K$-prime divisor $Y$ is contained in $V(f)$ if and only if $\nu_Y(f) > 0$. Now, (ii) becomes a special case of (i). We may assume that $X$ is affine. Let $U \subseteq X$ be open and affine. If $U = X$, then the complement is the support of the Weil divisor $0$, so we may suppose $U \neq X$. Then $X \setminus U = V(\mathfrak{a})$ with some proper ideal $\mathfrak{a} \neq 0$ because $U \neq \emptyset$. We claim that $X \setminus U$ contains only finitely many prime divisors. Indeed, for any $0 \neq f \in \mathfrak{a}$ the prime divisors $Y$ in $X \setminus U$ satisfy $\nu_Y(f) > 0$ and hence their number is finite. Let $U'$ be the complement of all closures of prime divisors contained in $X \setminus U$. Then $U \subseteq U'$ and both sets contain the same prime divisors. Since $X$ is a $K$-Krull scheme this implies $\mathcal{O}(U) = \mathcal{O}(U')$. 
%   Now, if $X_g$ is a principal subset contained in $U'$, then $\mathcal{O}(X_g) \rightarrow \mathcal{O}(U_{g_{|U}})$ is injective because all restricition maps on $K$-integral schemes are injective and surjective because $X_g$ and $U_{g_{|U}}$ contain the same $K$-prime divisors. Thus, $X_g \subseteq U$ and we conclude $U = U'$. 
%  \end{proof}
\end{proposition}

Next, we introduce good quotients of graded schemes. 

\begin{definition}\label{def:graded-good-quotients}
A morphism from a $K$-graded scheme to a scheme is called {\em $K$-invariant}. 
 A {\em good quotient by $K$} is an affine $K$-invariant morphism $q: X \rightarrow Y$ such that the pullback $\mathcal{O}_Y \rightarrow (q_*\mathcal{O}_X)_0$ is an isomorphism. 
\end{definition}

\begin{proposition}\label{prop:graded-good-quotients}
 If $q: X \rightarrow Y$ is a good quotient, then the following hold: 
 \begin{enumerate}
  \item $q$ is surjective and closed, 
  \item if $A_i \subseteq X, i \in I$ are closed then $q(\bigcap_i{A_i}) = \bigcap_i{q(A_i)}$, 
  \item every preimage $q^{-1}(y)$ contains a unique element which is contained in all closures of elements of $q^{-1}(y)$; and $y$ is a closed point if and only if this element is a closed point. 
 \end{enumerate}
 \begin{proof}
 For closedness we use the fact that $\langle \mathfrak{a}_0 \rangle_R \cap R_0 = \mathfrak{a}_0$ holds for $\mathfrak{a}_0 \trianglelefteq R_0$. 
 For (ii) we use $(\sum_i{\mathfrak{a}_i}) \cap R_0 = \sum_i{(\mathfrak{a}_i \cap R_0)}$. 
 Surjectivity and (iii) follow from the fact that for a prime ideal $\mathfrak{q}$ of $R_0$,there exists a unique $K$-prime $\mathfrak{p}$ that is maximal with $\mathfrak{p} \cap R_0 = \mathfrak{q}$, and $\mathfrak{q}$ is maximal if and only if $\mathfrak{p}$ is $K$-maximal. 
 Indeed, if $\mathfrak{q}$ is maximal then the sum $\mathfrak{m}$ over all $K$-graded ideals $\mathfrak{a}$ of $R$ with $\mathfrak{a} \cap R_0 = \mathfrak{q}$ has the desired properties. The general case reduces to this case via localization. 
 \end{proof}
\end{proposition}

For the $\ZZ$-graded case, properties (i) and (iii) were observed in \cite[Lemma 1.1.2]{Ca2006}. 
A good quotient that is bijective, i.e. a homeomorphism, is called {\em geometric}.
 A well-known example for a graded geometric quotient is the proj construction of a $\ZZ$-graded ring. 

 \begin{example}\label{ex:proj}
  Let $R$ be a $\ZZ$-graded ring with $\deg_{\ZZ}(R^+) \subseteq \ZZ_{\geq 0}$ and consider the proper $\ZZ$-graded ideal $\mathfrak{a} = \bigoplus_{n > 0}{R_n}$. Let $\b{X} := \Spec_{\ZZ}(R)$ and set $\rq{X} := \b{X} \setminus V(\mathfrak{a})$. Then the quotients $\b{X}_f \rightarrow \Spec((R_f)_0)$ for $f \in \mathfrak{a} \cap R^+$ glue to a good quotient $q: \rq{X} \rightarrow {\rm Proj}(R)$ by $\ZZ$ which is even bijective, i.e. geometric. 
  However, unless $R = R_0$, the structure sheaves of $\rq{X}$ and ${\rm Proj}(R)$ will be different. 
 \end{example}

$K$-graded schemes also occur naturally as {\em relative $K$-spectra} of a quasi-coherent $K$-graded sheaves on schemes: 

\begin{construction}
 Let $X$ be a scheme and let $\mathcal{R}$ be a quasi-coherent $K$-graded $\mathcal{O}_X$-algebra. 
Then $\Spec_K(\mathcal{R}(U))$ is open in $\Spec_K(\mathcal{R}(V))$ for any two affine open $U \subseteq V \subseteq X$ and hence the $K$-prime spectra of $\mathcal{R}(U)$ for all affine $U$ glue to a $K$-graded scheme, called 
 the {\em relative $K$-spectrum} $\Spec_{K,X}(\mathcal{R})$ of $\mathcal{R}$, and there is a commutative diagram 
 \[
 \xy
 \xymatrix{
 \Spec_X(\mathcal{R}) \ar[rr]\ar[dr] & & X \\
 & \Spec_{K,X}(\mathcal{R}) \ar[ur]_{q} &  
 }
 \endxy
\]
and $q_*\mathcal{O}_{\Spec_{K,X}(\mathcal{R})} = \mathcal{R}$ holds. If $\mathcal{R}_0 = \mathcal{O}_X$, then $q$ is a good quotient. 
\end{construction}

\begin{remark}\ref{rem:irredundancy}
 Let $L \leq \WDiv(X)$ be any subgroup mapping onto $\Cl(X)$ and let $\t{X} := \Spec_{L, X}(\mathcal{O}_X(L))$. For any Cox sheaf $\mathcal{R}$, every CIE $\pi : \mathcal{O}_X(L) \rightarrow \mathcal{R}$ induces a graded homeomorphism $\rq{X} := \Spec_{\Cl(X), X}(\mathcal{R}) \rightarrow \t{X}$ which is an isomorphism if and only if $L$ maps isomorphically onto $\Cl(X)$ (which in turn can only occur if $\Cl(X)$ is free). 
\end{remark}

\section{Proofs of Theorems~\ref{intro-thm:char-of-graded-char-spaces} and ~\ref{intro-thm:char-of-Cox-rings}}\label{graded-char-spaces}

\begin{proof}[Proof of Theorem~\ref{intro-thm:char-of-graded-char-spaces}]
First suppose that conditions (i)-(iii) hold. Consider the cover of $X$ by all affine open $U$. Then $\rq{X}$ is covered by their preimages $\rq{U}= q^{-1}(U)$ and this cover has a finite subcover. Consequently, $X$ has a finite affine cover and each $\mathcal{O}(U) = \mathcal{O}(\rq{U})_0$ is a Krull ring because $\mathcal{O}(\rq{U})$ is a $K$-Krull ring. Thus, $X$ is a Krull scheme. 
 We show that $\mathcal{R}:= q_* \mathcal{O}_{\rq{X}}$ is a Cox sheaf by applying Theorem~\ref{intro-thm:char-of-Cox-sheaves}. 
 Firstly, since $q$ is affine and hence $\mathcal{R}$ is quasi-coherent, we obtain 
 $$\mathcal{R}_{\xi} = Q^+(q_*\mathcal{O}_{\rq{X}}(U)) = Q^+(\mathcal{O}_{\rq{X}}(q^{-1}(U))) = \mathcal{O}_{\rq{X},\rq{\xi}}$$ 
 for any affine open $U$. Thus, $\mathcal{R}_{\xi}$ is $K$-simple and $\deg_K(\mathcal{R}_{\xi}^+) =  K$. Because $q$ is a good quotient, we also have $(\mathcal{R}_{\xi})_0 = ((q_*\mathcal{O}_{\rq{X}})_0)_{\xi} = \mathcal{O}_{X,\xi}$ and thus the constant sheaf $\mathcal{S} := \underline{\mathcal{R}_{\xi}}$ has all properties asserted in Theorem~\ref{intro-thm:char-of-Cox-sheaves}(i).
 Furthermore, there is a canonical isomorphism $q^* : \mathcal{S} \cong q_*\mathcal{K}_K$. 
 
 For the verification of Theorem~\ref{intro-thm:char-of-Cox-sheaves}(ii), consider a $K$-prime divisor $\rq{Y}$ and its image $Y:=q(\rq{Y})$. Due to (ii), we have $\imath_Y(\ZZ) = q_*\imath_{\rq{Y}}(\ZZ)$ and $\WDiv = q_*\WDiv^K$, and $\mu_Y:= q_*\nu_{\rq{Y}}: \mathcal{S}^+ \rightarrow \imath_Y(\ZZ)$ restricts to $\nu_Y$ on $\mathcal{K}$.  By definition, the family $\{\mu_Y\}_Y$ defines $\mathcal{R}$ as a $K$-Krull sheaf in $\mathcal{S}$. 
 
 For Theorem~\ref{intro-thm:char-of-Cox-sheaves}(iii), we first observe that $\div_{K,X} := \sum_Y{\mu_{Y,X}}= q_*\div_X^K$ is surjective because $\div_{\rq{X}}^K$ is so, due to $\Cl^K(\rq{X}) =0$. Furthermore, we have $\mathcal{R}(X)^{+,*} = \mathcal{R}(X)_0^*$ and thus Theorem~\ref{intro-thm:char-of-Cox-sheaves} implies that $\mathcal{R}$ is a Cox sheaf on $X$. 
 
 Now suppose that $\mathcal{R}$ is a Cox sheaf and $q : \rq{X} = \Spec_{K,X}(\mathcal{R}) \rightarrow X$ its characteristic space. Then the canonical isomorphism $q^*: \mathcal{R} \cong q_*\mathcal{O}_{\rq{X}}$ induces an isomorphism $q^*: \mathcal{S}\cong q_*\mathcal{K}_K$ as in the first part of the proof. By construction the $K$-spectra $\rq{U}$ of $\mathcal{R}(U)$ for affine $U \subseteq X$ form a cover of $\rq{X}$ by affine $K$-Krull schemes. 
 Since $\{\mu_Y\}_Y$ are the essential $K$-valuations of $\mathcal{R}$ and they restrict to $\{\nu_Y\}_Y$, there are natural bijections 
 \[
  \alpha: \bangle{\mu_{Y,U}^{-1}(\ZZ_{>0}) \cap \mathcal{O}(\rq{U})^+} \mapsto  \bangle{\nu_{Y,U}^{-1}(\ZZ_{>0}) \cap \mathcal{O}(U)} = Y
 \]
 between the $K$-prime divisors of $\rq{U} = \Spec_K(\mathcal{R}(U))$ and the prime divisors of $U$. 
Because $g$ is affine, these glue to an isomorphism 
$$\alpha : \WDiv^K(\rq{X}) \rightarrow \WDiv(X), \rq{Y} \mapsto q(\rq{Y})$$ 
which in turn induces an isomorphisms of presheaves $\beta: q_*\WDiv^K \cong \WDiv$ and $\beta_{\rq{Y}} :  q_*\imath_{\rq{Y}}(\ZZ) \cong \imath_{q(\rq{Y})}(\ZZ)$. By construction, we have $\beta_{\rq{Y}} \circ q_*\nu_{\rq{Y}} \circ q^* = \mu_{q(\rq{Y})}$ which gives $\beta \circ q_*\div^K \circ q^* = \div_K$,  the first supplement.
In particular, 
$\beta \circ q_*\div^K \circ q_{|\mathcal{K}^*}^* = \div$ holds and assertions (i)-(iii) are verified.

For the second supplement, consider a $K$-prime divisor $\rq{Y}$ and set $Y := q(\rq{Y})$. We claim that $q(\b{\{\rq{Y}\}}) = \b{\{Y\}}$. Indeed, $q(\b{\{\rq{Y}\}})$ is closed, irreducible and contains $\b{\{Y\}}$. The other inclusion follows because $\b{\{Y\}}$ has codimension one and $q(\b{\{\rq{Y}\}})$ is a proper subset of $X$. 
For any point $x \in X$ let $\rq{x} \in \rq{X}$ be the unique point contained in all closures of points mapped to $x$. Firstly, we claim that $\rq{x} \in \b{\{\rq{Y}\}}$. If $\rq{x} \in \b{\{\rq{Y}\}}$ then $x \in Y$ by the previous considerations. Conversely, if $x \in \b{\{Y\}} = q(\b{\{\rq{Y}\}})$ then $\b{\{\rq{Y}\}}$ contains a point $z$ with $q(z) = x$. By definition, $\rq{x} \in \b{\{z\}} \subseteq \b{\{\rq{Y}\}}$. 
Thus, we obtain
\[
 \mathcal{R}_x = \bigcap_{x \in \b{\{Y\}}}{\mathcal{S}(X)_{\mu_{Y,X}}} = \bigcap_{\rq{x} \in \b{\{\rq{Y}\}}}{{\mathcal{K}_K(\rq{X})}_{\nu_{\rq{Y},\rq{X}}}} = \mathcal{O}_{\rq{X},\rq{x}} 
\]
in $\mathcal{R}_{\xi} = \mathcal{O}_{\rq{X},\rq{\xi}}$. In particular, for $Y = x$ the point $\rq{x}$ lies in the closure of $\{\rq{Y}\}$ and $\mathcal{O}_{\rq{X},\rq{x}} = \mathcal{R}_Y = \mathcal{O}_{\rq{X},\rq{Y}}$ which implies $\rq{x} = \rq{Y}$. 
\end{proof}

The following Lemma was needed to show graded locality of the stalks $\mathcal{O}_X(L)_x$ and $\mathcal{R}_x$ in the respective proofs. 

\begin{lemma}\label{lem:div_K-and-sums}
 Let $\mathcal{R}$ be a Cox sheaf on $X$ and let $f, f' \in \mathcal{R}(U)_{[D]}$ with $f + f' \neq 0$. Then 
 \[
  |\div_{K,U}(f)|\cap |\div_{K,U}(f')| \subseteq |\div_{K,U}(f+f')| .
 \]
Accordingly, for $g, g' \in \mathcal{O}_X(D)(U)$ with $g + g' \neq 0$ we have
 \[
  |\div_U(g) + D| \cap |\div_U(g') + D| \subseteq |\div_U(g +g') + D| .
 \]
\begin{proof}
 It suffices to consider the case that $U$ is affine. Using Proposition~\ref{prop:divisor-complements} we calculate 
 \begin{align*}
  |\div_{K,U}(f)| \cap |\div_{K,U}(f')| & = q(|\div_{q^{-1}(U)}^K(q^*(f))|) \cap q(|\div_{q^{-1}(U)}^K(q^*(f'))|) \\
 & = q(|\div_{q^{-1}(U)}^K(q^*(f))| \cap |\div_{q^{-1}(U)}^K(q^*(f'))|) \\
 & = q(V_{q^{-1}(U)}(q^*(f)) \cap V_{q^{-1}(U)}(q^*(f'))) \\
 & \subseteq q(V_{q^{-1}(U)}(q^*(f+f')))  = |\div_{K,U}(f+f')| .
 \end{align*}
\end{proof}
\end{lemma}

\begin{proof}[Proof of Theorem~\ref{intro-thm:char-of-Cox-rings}]
First, let $R = \mathcal{R}(X)$ be the Cox ring of $X$. Then (i), (ii) and (iii) follow from Theorem~\ref{intro-thm:char-of-Cox-sheaves}. For assertion (iv) we additionally suppose that $X$ has an affine cover by complements of Weil divisors. Let $\mathfrak{p}$ be a $K$-prime divisor. By $K$-factoriality we have $\mathfrak{p} = \bangle{f}$ with some $K$-prime $f \in R$. Then $Y := \div_{K,X}(f)$ is contained in some affine open set $U = X \setminus |D|$ with some effective Weil divisor $D$. Let $\mathcal{O}_X(L)$ map onto $\mathcal{R}$. Then there exists $g \in \mathcal{O}_X(L)(X)$ with $D = \div_{L,X}(g) = \div_{K,X}(\pi_X(g))$. Let $h := \pi_X(g)$. Since $\div_{K,X}(f) \nleq \div_{K,X}(h)$, $f$ does not divide $h$ which means $h \notin \mathfrak{p}$. Hence, $R_{\mathfrak{p}} = (R_h)_{\mathfrak{p}}$ and the second ring has units in every $K$-degree by Theorem~\ref{thm:alg-divisorial-algebras} (iii) and Proposition~\ref{prop:compw-iso-epis} (iv). 

Now, let $K$ be finitely generated, and $R$ an algebraic Cox ring, i.e. a $K$-graded ring satisfying conditions (i), (ii) and (iv). We claim that there is a set of $K$-primes $f_1,\ldots,f_r$ such that $\bangle{\deg_K(f_j); j\neq k} = K$ for every $k=1,\ldots,r$.  
 Since $K$ is finitely generated, there are pairwise non-associated $K$-primes $f_1,\ldots, f_m$ with $\bangle{\deg_K(f_1),\ldots,\deg_K(f_m)}  = K$. The localization $R_{\bangle{f_1}}$ has units in every degree by (iv), so there are fractions $g_1/h_1,\ldots, g_n/h_n$, where none of the $g_j, h_j$ are divisible by $f_1$, whose degrees together generate $K$. Decomposing gives $f_{m+1},\ldots, f_t$ such that $f_1,\ldots,f_t$ are pairwise non-associated $K$-primes and $\bangle{\deg_K(f_2),\ldots,\deg_K(f_t)}= K$. Proceeding in this way for $k=2,\ldots,m$, we arrive at a set $f_1,\ldots,f_r$ with the requested properties. 
 
 For $j=1,\ldots,r$ let $R_j$ be the localization by the product of all $f_k$ with $k\neq j$. 
 Let $\rq{X}$ be the $K$-graded scheme with affine charts $\rq{X}_j:=\Spec_K(R_j)$. 
 Then by choice of $f_1,\ldots,f_r$ all $X_j = \Spec((R_j)_0)$ contain $X' = \Spec((R_{f_1\cdots f_r})_0)$ as a principal open subset and thus glue to a scheme $X$. The maps $\rq{X}_j \rightarrow X_j$ glue to a good quotient $q: \rq{X} \rightarrow X$ by $K$. 
 
 We verify that $R$ is the Cox ring of $X$ by showing that $q$ is a graded characteristic space. $\rq{X}$ is a $K$-Krull scheme because every $R_j$ is a $K$-Krull ring (they are even $K$-factorial).  
 Each $\rq{X}_j$ contains all $K$-prime divisors of $\b{X}:=\Spec_K(R)$ except the $K$-principal divisors of those $f_k$ with $k \neq j$. Thus, $\rq{X}$ contains every $K$-prime divisor of $\b{X}$, which implies $\mathcal{O}(\rq{X}) = R$ and $\mathcal{O}(\rq{X})^{+,*} = \mathcal{O}(\rq{X})_0^*$ by (ii). 
 Moreover, $\mathcal{K}^K(\rq{X}) = \mathcal{K}^K(\rq{X}_j) = Q^+(R)$ and hence $\Cl^K(\rq{X}) = \Cl(R) =0$ by (i).
 
 By (iv), each $R_j$ has units of every $K$-degree. Firstly, this yields $Q^+(R_j)_0 = Q((R_j)_0)$ and $\deg_K(R_j) = K$. There $q^* : \mathcal{K} \rightarrow (\mathcal{K}_K)_0$ is an isomorphism and $\deg_K(\mathcal{K}_K^+) = K$. 
 Secondly, there are bijections respecting sums, products, intersections and inclusions between the $K$-graded ideals of $R_j$ and the ideals of $(R_j)_0$. 
 In particular, there is an isomorphism $\WDiv^K(\rq{X}_j) \rightarrow \WDiv(X_j), \rq{Y} \mapsto g(\rq{Y})$ and for every $K$-prime divisor $\rq{Y} \in \rq{X}_j$ the generator of the maximal ideal of $\mathcal{O}_{X_j, q(\rq{Y})}$ also generates the $K$-maximal ideal of $\mathcal{O}_{\rq{X}_j,\rq{Y}}$. 
 Consequently, the induced isomorphism $\alpha: \WDiv^K(\rq{X}) \rightarrow \WDiv(X)$ satisfies $\alpha \circ \div^K \circ q^* = \div$.
Thus, Theorem~\ref{intro-thm:char-of-graded-char-spaces} gives the assertion. 
\end{proof}

\begin{remark}\label{rem:irredundancy}
 Conditions (i), (ii) and (iv) irredundantly characterize Cox rings of Krull schemes with affine cover by divisor complements and finitely generated class group. Indeed, \cite{Be2012} offers examples of rings satisfying (i) and (ii) but not (iv). Also, one may take any $K$-graded Cox ring $R$ and extend the grading trivially to a $K \oplus \ZZ$-grading. The monoid $R^+$ stays the same and its units remain in degree $0$, but $\deg_K(R^+) = K \oplus 0 \neq K \oplus \ZZ$, so the units of $R_{\mathfrak{p}}$ cannot attain degrees outside of $K \oplus 0$. 
 
 Examples with (i) and (iv) but not (ii) are obtained in Example~\ref{ex:A(K,phi)} whenever $\Cl(A)$ is not free and $\phi:= \id_{Div(A)}$ is chosen. 
 Examples with (ii) and (iv) but not (i) are also obtained from Example~\ref{ex:A(K,phi)} whenever $\Cl(A)$ is free and $\phi$ is the inclusion of a subgroup $K \subset Div(A)$ which maps injectively but not surjectively to $\Cl(A)$.
\end{remark}

\section{graded schemes and diagonalizable actions}\label{graded-schemes-and-diagonalizable-actions}

The defining algebraic data of a graded scheme $(X, \mathcal{O}_X)$ also define a scheme: 
By equipping each $\mathcal{O}_X(U)$ with the trivial $0$-grading, one obtains a $0$-graded quasi-coherent $\mathcal{O}_X$-algebra $\mathcal{O}_X^{(0)}$ whose $0$-graded relative spectrum $X^{(0)} := \Spec_X(\mathcal{O}_X^{(0)})$ is a scheme. The canonical affine morphism $X^{(0)} \rightarrow X$ which on affine charts is given as $\Spec(R) \rightarrow \Spec_K(R), \mathfrak{p} \mapsto \bangle{\mathfrak{p} \cap R^+}$ is surjective. 
 
 In this section, let $\KK$ be an algebraically closed field. 
 We will show that the functor $\mathfrak{f}: X \mapsto X^{(0)}$ induces an equivalence between the category $\mathcal{A}$ of graded reduced schemes $X$ of finite type over $\KK$ with finitely generated grading groups $gr(X)$ and the category $\mathcal{B}$ of prevarieties over $\KK$ with quasi-torus actions admitting affine invariant covers. 
 
\begin{definition}
Let $Z$ be a prevariety over $\KK$ with the action of a diagonalizable group $H:=\Spec_{\max}(\KK[K])$. Then $Z$ has an induced {\em $H$-invariant topology} $\Omega_{Z,H}$ consisting of the $H$-invariant Zariski open sets. The $K$-graded sheaf of rings obtained by restricting $\mathcal{O}_Z$ to $\Omega_{Z,H}$ is denoted $\mathcal{O}_{Z,H}$ and called the {\em $H$-invariant structure sheaf}. 
\end{definition}

If $Z$ is affine, then the $\Omega_{Z,H}$-closed (and $\Omega_{Z,H}$-irreducible) subsets are precisely the Zariski closed sets with $K$-homogeneous ($K$-prime) vanishing ideal. Equivalently, a subset of $Z$ is $\Omega_{Z,H}$-closed and $\Omega_{Z,H}$-irreducible if it is $\Omega_Z$-closed and its $\Omega_Z$-irreducible components are permuted by $H$.

\begin{proposition}
Let $\mathfrak{t}$ denote the functor sending a ringed space to its space of closed irreducible subsets with induced structure sheaf. 
Let $\mathfrak{g}$ be the equivalence from schemes of finite type over $\KK$ to prevarieties over $\KK$. 
Then we have an equivalence of categories
\begin{align*}
 \mathcal{A} & \longrightarrow \mathcal{B} \\
 \mathfrak{r}: X & \longmapsto \mathfrak{g}(\mathfrak{f}(X)) = \Spec_{\max, X}(\mathcal{O}_X^{(0)}) \\
 \mathfrak{t}(Z, \Omega_{Z,H},\mathcal{O}_{Z,H}) & \longmapsfrom [H \times Z \rightarrow Z] : \mathfrak{s} 
\end{align*}
Here $\mathfrak{r}(X)$ comes with the action by $\Spec_{\max}(\KK[gr(X)])$ which is induced on affine charts by the map $R \rightarrow \KK[gr(X)] \otimes_{\KK} R, f_w \mapsto \chi^w \otimes f_w$. 
\end{proposition}

\begin{remark}\label{remark:orbit-closures}
Let $Z$ be an affine $H$-variety. Then $H$-orbits correspond naturally to $K$-prime ideals of $R := \mathcal{O}(Z)$ of the form $\bangle{\mathfrak{m} \cap R^+}$ where $\mathfrak{m}$ is a maximal ideal of $R$. Let $X:= \Spec_K(R)$ and for a point $z \in Z$ set $S_z :=  \deg_K(R^+\setminus I(H z))$.
The set $V_X(I(H z))$ of $K$-prime ideals containing $I(H z)$ fits into the known correspondence (e.g. \cite[Prop. 3.8]{Ha2008}) between the orbits contained in $\b{H z}\cong \Spec_{\max}(\KK[S_z])$ and the faces of $S_z$ in the following way: 
\[
\begin{array}{r@{\hspace{3pt}}c@{\hspace{3pt}}c@{\hspace{3pt}}c@{\hspace{3pt}}l}
 {\rm orbits}(\b{H z}) & \longleftrightarrow & V_X(I(H z)) & \longleftrightarrow & \faces(S_z) \\
 H z_0 & \longmapsto & I(H z_0) ,\;\;\; \mathfrak{p} & \longmapsto & \deg_K(R^+ \setminus \mathfrak{p}) \\
 O_{\mathfrak{p}} & \longmapsfrom & \mathfrak{p} ,\quad \quad \quad \mathfrak{p}_{\tau} & \longmapsfrom & \tau 
\end{array}
\]
where $O_{\mathfrak{p}} := V_Z(\mathfrak{p}) \setminus \bigcup_{\mathfrak{p} \neq \mathfrak{q} \in V_X(\mathfrak{p})}{V_Z(\mathfrak{q})}$ and $\mathfrak{p}_{\tau} := I(H z) + \sum_{w \in S_z \setminus \tau}{R_w}$. 
The above bijections are order-reversing, where the order on orbits is defined as $H z_0 < H z_1 : \Leftrightarrow H z_0 \subseteq \b{H z_1}$ and the other sets are ordered by inclusion. 
\end{remark}

For a $\Omega_{Z,H}$-closed-irreducible subset $Y \subseteq Z$ the stalk of $\mathcal{O}_{Z,H}$ at $Y$ is defined as 
\[
 (\mathcal{O}_{Z,H})_Y := \varinjlim_{\substack{U \in \Omega_{Z,H} \\ U \cap Y \neq \emptyset}}{\mathcal{O}_Z(U)} .
\]
It coincides with the stalk of $\mathcal{O}_{\mathfrak{s}(Z)}$ at the point $Y \in \mathfrak{s}(Z)$. 
If $Z$ is $\Omega_{Z,H}$-irreducible then we denote by $\mathcal{K}_H$ the constant sheaf assigning the stalk at $Z$.

\begin{remark}\label{rem:inv-stalks-and-free-actions}
 $\Omega_{Z,H}$ has a basis of affine $H$-invariant open sets if and only if $Z \in \mathcal{B}$ (e.g. $Z$ allows a good quotient by $H$). 
In this case, the following statements hold: 
 \begin{enumerate}
  \item The stalk $(\mathcal{O}_{Z,H})_Y$ at $Y$ coincides with the graded localization  $\mathcal{O}(U)_{I(Y)}$ for every affine invariant open set $U$ meeting $Y$. In particular, if $Z$ is $\Omega_{Z,H}$-irreducible, then $\mathcal{K}_H$ is a $K$-simple sheaf. 
  \item The generic isotropy group of an $\Omega_{Z,H}$-closed $\Omega_{Z,H}$-irreducible set $Y$ is $\Spec_{\max}(\KK[K / \deg_K((\mathcal{O}_{Z,H})_Y^{+,*})])$. In particular, $H$ acts freely on a big open subset of $Z$ 
  if and only if $H$ acts freely on non-empty open subsets of all $H$-prime divisors, i.e. 
  if and only if the stalks $(\mathcal{O}_{Z,H})_Y$ at all $H$-prime divisors have units in every degree. 
 \end{enumerate}
 
\end{remark}

The graded schemes in $\mathcal{A}$ are of finite type, in particular $K$-noetherian, hence they have a cover by $K$-spectra of $K$-Krull rings if and only if they are {\em $K$-normal}, i.e. they have a cover by $K$-spectra of $K$-normal rings. Correspondingly, a $H$-prevariety $Z \in \mathcal{B}$ is called {\em $H$-normal} if and only if the sections of $\mathcal{O}_{Z,H}$ over affine invariant subsets are $\Chi(H)$-normal. A {\em $H$-prime divisor} is a $\Omega_{Z,H}$-closed $\Omega_{Z,H}$-irreducible subset $Y$ which is maximal among the proper subsets of $Z$ with these properties. An equivalent condition is that $Y$ is closed and the $\Omega_Z$-irreducible components of $Y$ are $1$-codimensional and are permuted by $H$, compare \cite[Sect. I.6.4]{ArDeHaLa}. 
Each $H$-prime divisor $Y$ on a $H$-normal prevariety $Z$ defines a discrete value sheaf $\imath_Y(\ZZ)$ on $\Omega_{Z,H}$, which takes values $\ZZ$ if $U$ intersects $Y$ and $0$ otherwise, and a discrete $K$-valuation $\nu_Y : \mathcal{K}_H^+ \rightarrow \imath_Y(\ZZ)$. These define the $K$-Krull sheaf $\mathcal{O}_{Z,H} = \bigcap_Y{(\mathcal{K}_H)_{\nu_Y}} \subseteq \mathcal{K}_H$. Their sum defines a morphism 
\[
 \div^H := \sum_Y{\nu_Y} : \mathcal{K}_H^+ \rightarrow \WDiv^H := \bigoplus_Y{\imath_Y(\ZZ)}
\]
to the presheaf of $H$-Weil divisors. Image and cokernel presheaf are the {\em $H$-principal divisors} $\PDiv^H$ resp. the {\em $H$-class group} $\Cl^H$. 
In this terminology, Theorem~\ref{intro-thm:char-of-graded-char-spaces} translates into the following characterization of characteristic spaces $\Spec_Z(\mathcal{R}) \rightarrow Z$ of Cox sheaves of finite type.

\begin{theorem}\label{thm:char-of-char-spaces-of-prevars}
 Let $q: \rq{Z} \rightarrow Z$ be a $H$-invariant morphism of prevarieties. 
 Then $Z$ is normal and $q$ is a characteristic space 
 if and only if the following hold:
 \begin{enumerate}
  \item $\rq{Z}$ is $H$-normal with $\deg_K(\mathcal{K}_H(\rq{Z})^+) = K$, 
  \item $q$ is a good quotient and induces a commutative diagram of presheaves %on $X$
  \[
   \xy
   \xymatrix
   {
   \mathcal{K}^* \ar[r]^-{\div} \ar[d]^-{\cong}_-{q^*}  &  \WDiv \\
   (q_*\mathcal{K}_H)^*_0 \ar[r]^-{q_*\div^H} & q_*\WDiv^H \ar[u]_-{\substack{\; \\ \rq{Y} \mapsto q(\rq{Y}) }}^-{\cong}
   }
   \endxy
  \]
  \item $\Cl^H(\rq{Z}) = 0$, and $\mathcal{O}(\rq{Z})^{+,*} = \mathcal{O}(\rq{Z})_0^*$. 
\end{enumerate}
If $\rq{Z} = \Spec_Z(\mathcal{R})$ with a Cox sheaf $\mathcal{R}$ then with $\div_K := \sum_Y{\mu_Y}$ the following commutative diagram extends the diagram of (ii): 
  \[
   \xy
   \xymatrix
   {
   \mathcal{S}^+ \ar@{->>}[r]^-{\div_K} \ar[d]^-{\cong}_-{q^*}  &  \WDiv \\
   q_*\mathcal{K}_H^+ \ar@{->>}[r]^-{q_*\div^H} & q_*\WDiv^H \ar[u]_-{\substack{\; \\ \rq{Y} \mapsto q(\rq{Y}) }}^-{\cong}
   }
   \endxy
  \]
  Each prime divisor $Y$ is the image $q(\rq{Y})$ of a uniqe $H$-prime divisor $\rq{Y}$. 
  If $H\rq{z} \subseteq \rq{Z}$ is the unique closed orbit in $q^{-1}(z)$, then $H\rq{z} \subseteq \rq{Y}$ if and only if $z \in Y$. In particular, $(\mathcal{O}_{\rq{X},H})_{H\rq{z}} = \mathcal{R}_z$.
\end{theorem}

This result only required those properties of normal prevarieties which they share with Krull schemes. 
Using the fact that $X_{\reg}$ is big in $X$ and $\WDiv(X_{\reg}) = {\rm CaDiv}(X_{\reg})$ combined with noetherianity yields additional properties of characteristic spaces, e.g. irreducibility and normality, and the property that $q^{-1}(X_{\reg})$ is a big subset on which $H$ acts freely and $q$ is geometric, see \cite[Sect. I.5, I.6]{ArDeHaLa} for proofs and further properties of characteristic spaces. 

\begin{remark}\label{rem:toric-graded-schemes}
Recall that a lattice is a finitely generated free abelian group. 
By a separated {\em toric} graded scheme over $\KK$ we mean a quasi-compact separated graded scheme $X$ of finite type over $\KK$ such that 
\begin{enumerate}
 \item the grading group $M = gr(X)$ is a lattice, 
 \item $X$ is $M$-normal, 
 \item $X$ contains $\Spec_M(\KK[M])$ as an open subset, 
 \item $X$ is {\em effectively graded}, i.e. $\bangle{\deg_M(\mathcal{O}(U))} = M$ for all affine open $U \subseteq X$. 
\end{enumerate}
$Z \in \mathcal{B}$ is a (separated) toric variety if and only if $X = \mathfrak{s}(Z) \in \mathcal{A}$ is a separated toric graded scheme. 
If $\Sigma$ is the fan in $N = M^*$ describing $Z$ where $M = gr(X)$, then $\Omega_{Z, T}$ is finite and its basis consists of the affine invariant charts $\{Z_{\sigma}\}_{\sigma \in \Sigma}$. The $\Omega_{Z,T}$-irreducible $\Omega_{Z,T}$-closed subsets are the orbit closures $\{V(\sigma)\}_{\sigma \in \Sigma}$.   
For $\sigma \in \Sigma$ let $I_{\sigma} \trianglelefteq \KK[M \cap \sigma^{\vee}]$ be the vanishing ideal of the closed orbit of $Z_{\sigma}$. 
Then $\mathcal{O}_{X, I_{\sigma}} = \KK[M \cap \sigma^{\vee}]_{I_{\sigma}} = \KK[M \cap \sigma^{\vee}]$ 
and there is a natural bijection 
 \begin{align*}
  \Sigma & \longrightarrow X \\
  \sigma & \longmapsto I_{\sigma} \\
  \deg_M(\mathcal{O}_{X,p})^{\vee} & \longmapsfrom p 
 \end{align*}
 Furthermore, applying Example~\ref{ex:graded-spec-of-monoid-algebra} to each monoid $\sigma^{\vee} \cap M$ and its monoid algebra $\KK[\sigma^{\vee} \cap M]$, we see that $X$ is canonically homeomorphic to the $\mathbb{F}_1$-scheme $A_{\Sigma}$ obtained by gluing the spectra of $(\sigma^{\vee} \cap M)\sqcup \{\infty\}$. 
 For a cone $\sigma \in \Sigma$ we denote by $p_{\sigma}$ the closed point $(\sigma^{\vee} \cap M)\sqcup \{\infty\} \setminus (\sigma^{\perp} \cap M)$ of $\Spec((\sigma^{\vee} \cap M)\sqcup \{\infty\})$.  
 The canonical bijection between $A_{\Sigma}$ and $\Sigma$ is
 \begin{align*}
  \Sigma & \longrightarrow A_{\Sigma} \\
  \sigma & \longmapsto p_{\sigma} \\
  (\mathcal{O}_{A_{\Sigma},p} \setminus \{\infty\})^{\vee} & \longmapsfrom p 
 \end{align*}
 A different kind of connection between the categories of $\mathbb{F}_1$-schemes of finite type and toric varieties is established in \cite{De2008}. 
 \end{remark}

\end{document}